\documentclass{amsart}
\usepackage{amssymb, latexsym}
\usepackage{enumerate}
\usepackage{amsmath}
\usepackage{amsthm}
\usepackage{graphicx}
\usepackage{verbatim}
\usepackage{hyperref}
\usepackage{dsfont}
\usepackage[all]{xy}

\newtheorem{theorem}{Theorem}[section]
\newtheorem*{thm}{Theorem}
\newtheorem{proposition}[theorem]{Proposition}
\newtheorem{lemma}[theorem]{Lemma}
\newtheorem{corollary}[theorem]{Corollary}

\theoremstyle{definition}
\newtheorem{definition}[theorem]{Definition}
\newtheorem{example}[theorem]{Example}
\newtheorem{remark}[theorem]{Remark}

\newtheorem*{convention}{Convention}
\newtheorem{problem}{Open problem}

\newcommand{\enproof}{\hspace*{\stretch{1}}\qedsymbol}  

\newcommand{\norm}[1]{\left\Vert#1\right\Vert}  

\newcommand{\Lspace}{\mathit{L}}
\newcommand{\Lone}{\Lspace^1}
\newcommand{\Ltwo}{\Lspace^2}
\newcommand{\Linfty}{\Lspace^\infty}
\newcommand{\lspace}{\ell}
\newcommand{\lone}{\lspace^1}

\newcommand{\linfty}{\lspace^\infty}

\newcommand{\C}{\mathcal{C}}

\newcommand{\set}[1]{\left\{#1\right\}}  
\newcommand{\closure}[1]{\overline{#1}}  
  
\newcommand{\Fa}{\mathit{A}}
\newcommand{\FSa}{\mathit{B}}
\newcommand{\Measures}{\mathit{M}}
\newcommand{\VN}{\mathit{VN}}

\newcommand{\spectrum}{\sigma}

\newcommand{\complexs}{\mathbb{C}}    
\newcommand{\conjugate}[1]{\overline {#1}} 

\newcommand{\unitcircle}{\mathbb{T}}
\newcommand{\reals}{\mathbb{R}}    
\newcommand{\naturals}{\mathbb{N}}    
\newcommand{\integers}{\mathbb{Z}}    

\newcommand{\tuple}[1]{\boldsymbol{#1}}

\newcommand{\Acal}{\mathcal{A}}

\newcommand{\Fcal}{\mathcal{F}}
\newcommand{\Mcal}{\mathcal{M}}

\newcommand{\Hcal}{\mathcal{H}}

\newcommand{\Jcal}{\mathcal{J}}
\newcommand{\Ucal}{\mathcal{U}}

\newcommand {\Bcal}{{\mathcal B}}

\newcommand {\Mfrak}{{\mathfrak M}}
\newcommand {\Scal}{{\mathcal S}}
\newcommand {\Tcal}{{\mathcal T}}
\newcommand {\Tfrak}{{\mathfrak T}}

\newcommand{\abs}[1]{\left|#1\right|}

\newcommand{\dual}[1]{#1^\prime}
\newcommand{\bidual}[1]{#1^{\prime\prime}}
\newcommand{\tripledual}[1]{#1^{\prime\prime\prime}}
\newcommand{\predual}[1]{#1_*}

\newcommand{\closedball}[2]{#1_{[#2]}}

\newcommand{\cstar}{C$^*$}
\newcommand{\wstar}{W$^*$}
\newcommand{\duality}[2]{\left\langle #1,#2 \right\rangle}

\newcommand{\operators}{\mathcal{B}}

\newcommand{\id}{\mathrm{id}}

\newcommand{\inner}[2]{\left( #1|#2\right)}

\renewcommand{\Re}{\mathop{\rm Re}}

\newcommand{\init}{\dot{\imath}}
\newcommand{\emath}{\mathrm{e}}

\newcommand{\support}{\mathop{\rm supp\,}}
\newcommand{\dd}{\,{\rm d}}

\begin{document}

\title[Characterisations of Fourier and Fourier--Stieltjes algebras]
{Characterisations of Fourier and Fourier--Stieltjes algebras on locally compact groups} 

\author[A.\ T.-M.\ Lau]{Anthony To-Ming Lau}
\address{Department of Mathematical and Statistical Sciences\\
University of Alberta\\
Edmonton, AB T6G 2G1\\
Canada}
\email{anthonyt@ualberta.ca}
\thanks{The first author is supported by an NSERC--grant MS100.}

\author[H.\ L.\ Pham]{Hung Le Pham}
\address{School of Mathematics and Statistics\\
Victoria University of Wellington \\ Wellington 6140, New Zealand}
\email{hung.pham@vuw.ac.nz}
\thanks{The second author is supported by a Fast Start Marsden grant while undertaking part of this research.}

\date{}

\begin{abstract}
Motivated by the beautiful work of M. A. Rieffel (1965) and of M. E. Walter (1974), we obtain characterisations of the Fourier algebra $\Fa(G)$ of a locally compact group $G$ in terms of the class of $F$-algebras (i.e. a Banach algebra $A$ such that its dual $\dual{A}$ is a \wstar-algebra whose identity is multiplicative on $A$). 
For example, we show that the Fourier algebras are precisely those commutative semisimple $F$-algebras that are Tauberian, contain a nonzero real element, and possess a dual semigroup that acts transitively on their spectrums. Our characterisations fall into three flavours, where the first one will be the basis of the other two.     
The first flavour also implies a simple characterisation of when the predual of a Hopf-von Neumann algebra   is the Fourier algebra of a locally compact group. We also obtain similar characterisations of the Fourier--Stieltjes algebras of $G$. En route, we prove some new results on the problem of when a subalgebra of $\Fa(G)$ is the whole algebra and on representations of discrete groups.
\end{abstract}



\maketitle

\section{Introduction}

\noindent In \cite{Rieffel}, M. Rieffel characterises the group algebra $\Lone(\Gamma)$ of a locally compact \emph{abelian} group $\Gamma$ in terms of $L'$-inducing characters on a Banach algebra, which in turn based on his concept of $L$-inducing functional $\phi$ on a Banach space. In this paper, we extend this to the notion of noncommutative $L'$-inducing for a character $\phi$ on a Banach algebra and show in Theorem \ref{a characterisation of Fourier algebras 5} that if $A$ is a commutative semisimple Banach algebra such that (i) every character of $A$ is noncommutative $L'$-inducing, and (ii) $A$ is Tauberian and conjugation-closed, then $A\cong \Fa(G)$ (i.e. isometrically isomorphic) for a locally compact group $G$. In fact, a stronger result will be proved where the assumption that $A$ is conjugation-closed is replaced by a weaker one that $A$ contains a nonzero real element. 

A key step in the proof of this result is Theorem \ref{a characterisation of Fourier algebras} where we characterise $\Fa(G)$ as an $F$-algebra $A$ which is also a conjugation-closed Tauberian subalgebra of $\C(\Omega)$  for some topological space $\Omega$ with an additional group structure such that the left-translations by elements of $\Omega$ preserve $A$ and are collectively controlled by a natural inequality on their norms as operators on $A$. This theorem is also applied to obtain, in Theorem \ref{a characterisation of Fourier algebras 4}, a similar characterisation but now,  instead of a norming condition collectively imposed on all left-translation operators, each of these left-translations $L_s$ is assumed to satisfy individually the inequality
\begin{align*}
			\norm{f-\emath^{\init\theta}L_sf}^2+\norm{f+\emath^{\init\theta}L_sf}^2\le 4\norm{f}^2 \qquad(f\in A,\ \theta\in\reals)\,,
\end{align*}
which was discovered by M. Walter in \cite{Walter74} to characterise the left-translations and the right-translations of a Fourier algebra $\Fa(G)$ among its  isometric automorphisms. We remark here that although characterisation results in \S \ref{Noncommutative L'-inducing property and characterisations of Fourier algebras} do not require the full power of Theorem \ref{a characterisation of Fourier algebras}, our proofs of the characterisations in \S \ref{Special isometries and characterisations of Fourier algebras} do.

A noteworthy variation of Theorem \ref{a characterisation of Fourier algebras 4}, which is Theorem \ref{a characterisation of Fourier algebras 4b}, characterises $\Fa(G)$ as a commutative semisimple $F$-algebra that is also Tauberian, contains a nonzero real element, and possesses a dual semigroup that acts transitively on $\spectrum(A)$. Note that \emph{a dual semigroup} is a concept defined for any Banach algebra $A$ (cf. Definition \ref{dual semigroup}), a dual semigroup always exists and acts canonically on $\spectrum(A)$ by transposition.

We also give similar characterisations of the Fourier--Stieltjes algebra $\FSa(G)$ of a locally compact group $G$ in \S\ref{Fourier--Stieltjes algebras}, where we replace the Tauberian condition, hold by all Fourier algebras, by the Eberlein condition, hold by all Fourier--Stieltjes algebras; for details of these conditions, see \S\ref{Preliminaries}. However, in the most general situation, our characterisations will actually capture a whole class of subalgebras of $\FSa(G)$ that contain the reduced Fourier--Stieltjes algebra $\FSa_r(G)$,  and thus capture precisely $\FSa(G)$ only when $G$ is amenable.  

We remark that $\Fa(G)$ as well as $\FSa(G)$, despite also having an adjoint operation induced from the \cstar-algebraic structure of their dual, are actually $*$-algebras whose natural involution is the pointwise complex-conjugation, recalling the involution of $\Fa(G)=\Lone(\widehat{G})$ and of $\FSa(G)=\Measures(\widehat{G})$ when $G$ is abelian. Under this context, the conjugation-closed condition in some of our characterisations becomes natural. In some other characterisations, we could weaken this assumption to require a priory only that the algebra under consideration contains a nonzero real element/function. In fact, this could be relaxed even further to the requirement that our algebra \emph{approximately contains a nontrivial real element/function}, in the (rather weak) sense of Definition \ref{nontrivial real function} (see also Example \ref{nontrivial real function, example 1} for some common situations where this happens). In some special cases, we shall be able to remove these conditions completely. Tools to achieve these will be developed in \S\ref{Invariant subalgebras of Fourier algebras}.

Among the auxiliary results proved in \S\ref{Invariant subalgebras of Fourier algebras}, we note here two particular ones. The first, which is Corollary \ref{when A=A(G), compact connected component case}, implies that if $G$ is a locally compact group whose connected component of the identity is compact, and if $A$ is a closed translation-invariant subalgebra of $\Fa(G)$ with $\spectrum(A)=G$, then $A=\Fa(G)$. The second, which is Corollary \ref{weak*-closed invariant subalgebra of B(G) when G is discrete 2}, says that, for any amenable discrete group $G$, every weak$^*$-closed translation-invariant subalgebra of $\FSa(G)$ is conjugation-closed. In other words, any sub-semidual of an amenable discrete group $G$ is actually a subdual, which could be considered as a dual version to a well-known result that a closed sub-semigroup of a compact group must be a subgroup.

A Banach algebra $A$ is \emph{an $F$-algebra} if it is the predual of a \wstar-algebra $\Mfrak$ and the identity $1$ of $\Mfrak$ is a \emph{character} (i.e.  nonzero multiplicative linear functional) on $A$. This latter condition is equivalent to $P(A)$, the set of all positive normal functionals of $\Mfrak$, forms a semigroup under the multiplication of $A$. Examples of $F$-algebras include the predual algebra of a Hopf--von Neumann algebra, in particular, a quantum group algebra, the group algebra $\Lone(G)$ or the Fourier algebra $\Fa(G)$ of a locally compact group $G$. They also include the measure algebra $\Measures(S)$ of a locally compact semigroup $S$, and the Fourier--Stieltjes algebra $\FSa(G)$ of a topological group $G$. (See \cite{DLS}, \cite{Lau83}, \cite{LL2012}.) Moreover, the hypergroup algebra $\Lone(H)$ and the measure algebra $\Measures(H)$ of a locally compact hypergroup $H$ with a left Haar measure are $F$-algebras. In this case, Willson shows in \cite[Theorem 5.2.2]{Willson2011} (see also \cite[Remark 5.3]{Willson2014}) that $\dual{(\Lone(H))}=\Linfty(H)$ is not a Hopf--von Neumann algebra unless $H$ is a locally compact group.  $F$-algebras are also referred to as Lau algebras (see \cite{Pier}). 

Note that \emph{all of our characterisations are necessary and sufficient}; the other directions are simply omitted from the statements (but some of them are recorded in the text). Note also that, except for the results in \S\ref{Noncommutative L'-inducing property and characterisations of Fourier algebras} and the corresponding ones in \S\ref{Fourier--Stieltjes algebras} for which the abelian case has already been treated in \cite{Rieffel},  our results provide new characterisations of $\Fa(G)$ and $\FSa(G)$ (i.e. of $\Lone(\Gamma)$ and $\Measures(\Gamma)$) for the class of locally compact abelian groups; however, except for Corollary \ref{a characterisation of Fourier algebras 4b, abelian}, we shall not explicitly give the other statements for this special case.

\section{Preliminaries}
\label{Preliminaries}

\noindent In this paper, all compact or locally compact topologies are assumed to be Hausdorff. 

A subspace of complex-valued functions on a set $X$ is \emph{conjugation-closed} if it is closed under pointwise complex-conjugation. It is obvious that a conjugation-closed subspace $A$ of functions on $X$ always contains a nonzero real function unless $A=\set{0}$.

For a subset $X$ of a Banach space $E$ and $r>0$, we set $\closedball{X}{r}:=\set{x\in X\colon \norm{x}\le r}$.

Let $\Omega$ be a semigroup. 
The left translations of a function $f:\Omega\to \complexs$ are
\begin{align*}
	&L_sf:t\mapsto f(st)\,,\Omega\to \complexs\qquad(s\in \Omega)\,.
\end{align*}

Let $G$ be a locally compact group. For each continuous unitary representation $\pi$ of $G$, denote by $\C^*_\pi(G)$ the \emph{group \cstar-algebra associated with $\pi$}, i.e. the norm closure of $\set{\pi(f)\colon\ f\in \Lone(G)}$ in $\operators(\Hcal_\pi)$, where $\Hcal_\pi$ is the Hilbert space associated with $\pi$. Then $\C^*_\pi(G)$ is naturally a quotient of the \emph{full group \cstar-algebra $\C^*(G)$} of $G$. Another continuous unitary representation  $\rho$ of $G$ 
is said to be \emph{weakly contained} in $\pi$, and write $\rho\preccurlyeq\pi$, if $\C^*_\rho(G)$ is a quotient of $\C^*_\pi(G)$ in the natural setting as quotients of $\C^*(G)$ (see \cite[18.1.3]{Dixmier} or \cite[\S 1.6]{KT}). 

The dual of $\C^*_\pi(G)$, denoted by $\FSa_\pi(G)$, is identified naturally with the linear space of functions on $G$ spanned by the coefficient functions of all continuous unitary representations $\rho$ that are weakly contained in $\pi$. For each set $\Scal$ of (equivalent classes of) continuous unitary representations of $G$, $\C^*_\Scal(G)$ and $\FSa_\Scal(G)$ are defined as $\C^*_\pi(G)$ and $\FSa_\pi(G)$ where $\pi:=\bigoplus_{\rho\in\Scal}\rho$. The dual of $\C^*(G)$ is simply denoted as $\FSa(G)$ and it is the \emph{Fourier--Stieltjes algebra} of $G$. 
We also denote by $P(G)$ the set of continuous positive definite functions on $G$, and by $P_\Scal(G)$ the set of continuous positive definite functions belonging to $\FSa_\Scal(G)$.

The \emph{left regular representation} of $G$ is denoted by $\lambda_G$. The corresponding \cstar-algebra, the \emph{reduced group \cstar-algebra} of $G$, and its dual, the \emph{reduced Fourier--Stieltjes algebra} of $G$, are denoted by $\C^*_r(G)$ and $\FSa_r(G)$, respectively. The \emph{group von Neumann algebra} of $G$, denoted as $\VN(G)$, is the weak$^*$-closure of $\C^*_r(G)$ in $\operators(\Ltwo(G))$, and its predual  $\Fa(G)$, again identified naturally with a space of functions on $G$, is the \emph{Fourier algebra} of $G$.

For more details of the theory of the Fourier and Fourier--Stieltjes algebras on general locally compact groups, we refer the reader to the seminal paper \cite{Eymard} by Eymard where these objects were introduced.

Denote by $\widehat{G}$ \emph{the dual space} of $G$, which is the set of all equivalence classes of irreducible continuous unitary representations of $G$, endowed with Fell's topology; when $G$ is abelian, $\widehat{G}$ is just the dual group of $G$. For each (class of) continuous unitary representation $\pi$ of $G$, the \emph{support} of $\pi$, denoted by $\support\pi$, is the subset of $\widehat{G}$ consisting of those classes of representations that are weakly contained in $\pi$. A subset $\Scal$ of $\widehat{G}$ is a \emph{subdual} if $\support\left(\bigoplus_{\pi\in \Scal} \pi\right)\subseteq \Scal$, \footnote{In other words, $\Scal$ is a closed subset of $\widehat{G}$.} and if $\conjugate{\pi}\in \Scal$ and $\support (\pi\otimes\rho)\subseteq\Scal$ whenever $\pi,\rho\in\Scal$ (see \cite{Kaniuth}, but note that we do not require $\Scal$ to contain the trivial representation), where $\conjugate{\pi}$ is the conjugate representation of $\pi$. More generally, we shall call a subset $\Scal$ of $\widehat{G}$ is a \emph{sub-semidual} if $\support\left(\bigoplus_{\pi\in \Scal} \pi\right)\subseteq \Scal$  and if $\support (\pi\otimes\rho)\subseteq\Scal$ whenever $\pi,\rho\in\Scal$.

It is standard that the spaces of the form $\FSa_\Scal(G)$, for closed subsets $\Scal\subseteq \widehat{G}$, are precisely the weak$^*$-closed translation-invariant subspaces  of $\FSa(G)$. 
For such a space to be a subalgebra, simply because the linear combinations of products of the coefficient functions of two representations are precisely those of the coefficient functions of their tensor product, we need the following:

\begin{lemma}\label{algebra of sub-semidual}
For a closed subset $\Scal$ of $\widehat{G}$, $\FSa_\Scal(G)$ is an algebra on $G$ if and only if $\Scal$ is a sub-semidual of $G$. 
\end{lemma}

For a topological space $\Omega$, we shall write $\C(\Omega)$ the algebra of complex-valued continuous functions on $\Omega$, and write $\C_c(\Omega)$ its subalgebra consisting of those functions with compact supports in $\Omega$, while denote by $\C_0(\Omega)$ the subalgebra of $\C(\Omega)$ consisting of those functions that vanish at the infinity, it is actually a commutative \cstar-algebra whose natural norm is the uniform norm (and whose spectrum is $\Omega$ in the case where $\Omega$ is a  locally compact space).

Let $A$ be a commutative Banach algebra. We denote by $\spectrum(A)$ the \emph{spectrum} of $A$, which is a locally compact space with respect to the relative weak$^*$-topology of $\dual{A}$, and by $\hat{f}$, which is a function in $\C_0(\spectrum(A))$, the \emph{Gelfand transform} of $f$. The image of $A$ via the Gelfand transform is denoted as $\widehat{A}$, a subalgebra of $\C_0(\spectrum(A))$, which is also a Banach algebra under the norm induced from $A$. Thus we shall call $A$ \emph{conjugation-closed} if $\widehat{A}$ is, and call an element $f$ of $A$ \emph{a real element} if $\hat{f}$ is a real function.

For brevity, we shall use the following terminology. 

\begin{definition}\label{nontrivial real function}
Let $\Omega$ be a topological space. A subset $A\subseteq \C(\Omega)$ is said to \emph{approximately contain a nontrivial real function} if there is a bounded real function $\phi$ on $\Omega$ with the following properties:
\begin{enumerate}
	\item  for each compact subset $K\subseteq\Omega$, there exist a compact subset $K_0$ of $K$ and an $\alpha_K\in\reals$ such that, for every compact subset $L$ of $K\setminus K_0$, there is a sequence $(f_n)_{n=1}^\infty\subseteq A$ that converges pointwise to $\phi$ on $K_0\cup L$ and each $\abs{f_n}\le \alpha_K$ on $K$, and 
	\item there exists a downward directed collection $\Ucal$ of $A$-sharp open subsets of $\Omega$ such that $\phi(U)\cap \phi(\Omega\setminus \closure{U})=\emptyset$ for each $U\in \Ucal$ and that $\bigcap_{U\in \Ucal} \closure{U}$ is a nonempty compact set;
\end{enumerate}
here, $U$ is \emph{$A$-sharp} means that for every $s\in U$ and $t\in \Omega\setminus U$ there is $f\in A$ such that $f(s)\neq f(t)$, and so every subset of $\Omega$ is automatically $A$-sharp whenever $A$ separates points of $\Omega$.

A commutative Banach algebra $A$ \emph{approximately contains a nontrivial real element} if $\widehat{A}$ approximately contains a nontrivial real function.
\end{definition}

\begin{remark}
The technical condition (ii) in the definition above implies particularly that, for each $U\in \Ucal$, there is a set $E$ of values for $\phi$ such that $U\subseteq\phi^{-1}(E)\subseteq\closure{U}$. Condition (ii) is easily satisfied, for examples:
\begin{itemize}
	\item when $\phi\in\C_0(\Omega)\setminus\set{0}$; in this case, we can take $\Ucal$ to  consist of just a single open set $\set{t\in\Omega\colon \abs{\phi(t)}>\varepsilon}$ for any positive number $\varepsilon<\abs{\phi}_\Omega$, or 
	\item more generally, when $\phi\in \C(\Omega)$ having a compact level set, i.e. a set of the form $\phi^{-1}(r)$ for some value $r$ of $\phi$; in this case, we can take $\Ucal:=\set{\set{t\in\Omega\colon \abs{\phi(t)-r}<\varepsilon}\colon \varepsilon>0}$; 
\end{itemize}
the fact that these open sets are $A$-sharp is a consequence of condition (i) of Definition \ref{nontrivial real function}. Note that the approximation condition (i) implies that $\phi$ is the pointwise limit of some sequence in $A$ on each finite subset of $\Omega$. 
\end{remark}

The following examples show various simple situations where the property of approximate containment of a nontrivial real function holds:

\begin{example} \label{nontrivial real function, example 1}
Let $\Omega$ be a topological space, and let $A\subseteq \C(\Omega)$. Either:
\begin{enumerate}
	\item if $A\subseteq\C_0(\Omega)$ and the closure of $A$ in the uniform topology on $\Omega$ contains a nonzero real function $\phi$; or  
	\item more generally, if there is a real continuous bounded function $\phi$ with a compact level set that, on each compact subset of $\Omega$, is the pointwise limit of a bounded sequence in $A$; or 
	\item if $\Omega$ is locally compact, $A$ separates points of $\Omega$, and there is a compact subset $K$ of $\Omega$ with a nonempty interior and with the property that for any neighbourhood $W$ of $K$  there exists a function $f\in A$ such that $f=1$ on $K$, $f=0$ on $\Omega\setminus W$, and $\abs{f}\le 1$ on $W\setminus K$; 
\end{enumerate}
then $A$ approximately contains a nontrivial real function. 
\end{example}

\begin{lemma}\label{real element vs real function}
Let $A$ be a Banach algebra that is a subalgebra of $\C_0(\Omega)$ for some topological space $\Omega$. Suppose that every character of $A$ is implemented by an element of $\Omega$ and $A$ approximately contains a nontrivial real function on $\Omega$. Then the Banach algebra $A$ approximately contains a nontrivial real element. 
\end{lemma}
\begin{proof}
Denote by $\eta:\Omega\to\spectrum(A)$ the natural map, which is continuous. This map is also surjective by assumption. For each compact subset $L$ of $\spectrum(A)$,  since $A\subseteq \C_0(\Omega)$, it is easy to see that $\eta^{-1}(L)$ is compact. Thus $\eta$ is proper, and so it is closed \cite{Whyburn65}. Let $\phi$ be given as in Definition \ref{nontrivial real function} for $A$. Then since $\phi$ is the pointwise limit of some sequence in $A$ on each finite subset of $\Omega$  by condition (i) in Definition \ref{nontrivial real function}, it follows that  $\phi(t)=\phi(s)$ whenever $\eta(s)=\eta(t)$. Thus $\phi=\psi\circ\eta$ for a bounded real function $\psi$ on $\spectrum(A)$. The fact that $\psi$ satisfies condition (i) in Definition \ref{nontrivial real function} for $\widehat{A}$ is a bit tedious but easy to check. To prove condition (ii) in Definition \ref{nontrivial real function} for $\widehat{A}$ and $\psi$, let $\Ucal$ be  a collection of $A$-sharp open subsets of $\Omega$ such that $\phi(U)\cap \phi(\Omega\setminus \closure{U})=\emptyset$ for each $U\in \Ucal$ and that $L:=\bigcap_{U\in \Ucal} \closure{U}$ is a nonempty compact set, which exists since $A$ approximately contains $\phi$. Since each $U\in\Ucal$ is $A$-sharp, we see that $U=\eta^{-1}(\eta(U))$, and so, $\eta(U)$ is open. Since $\eta$ is closed, $\eta(\closure{U})=\closure{\eta(U)}$, and so 
\[
	\psi(\eta(U))\cap\psi(\spectrum(A)\setminus \closure{\eta(U)})\subseteq\psi(\eta(U))\cap\psi(\eta(\Omega\setminus\closure{U}))=\phi(U)\cap \phi(\Omega\setminus\closure{U})=\emptyset.
\]
Finally, we \emph{claim} that $\eta(L)=\bigcap_{U\in\Ucal}\closure{\eta(U)}$. Indeed, it is obvious that $\eta(L)\subseteq\bigcap_{U\in\Ucal}\closure{\eta(U)}$. Now take any $x\in\spectrum(A)\setminus \eta(L)$, and choose an open set $W\ni x$ such that $\closure{W}$ is compact and disjoint from $\eta(L)$. Then $\eta^{-1}(\closure{W})$ is compact and
\[
	\emptyset=\eta^{-1}(\closure{W})\cap L=\eta^{-1}(\closure{W})\cap \bigcap_{U\in \Ucal} \closure{U}\,.
\]
Thus, since $\Ucal$ is downward directed, there exists $U_0\in \Ucal$ such that
\[
	\emptyset=\eta^{-1}(\closure{W})\cap \closure{U_0}\supseteq\eta^{-1}(W)\cap U_0=\eta^{-1}(W)\cap  \eta^{-1}(\eta(U_0))=\eta^{-1}\left(W\cap  \eta(U_0)\right),
\]
and so, $W\cap \eta(U_0)=\emptyset$. This implies that $W\cap \closure{\eta(U_0)}=\emptyset$, and so in particular $x\notin \bigcap_{U\in\Ucal}\closure{\eta(U)}$. Hence, the claim is proved, and so $\set{\eta(U)\colon U\in\Ucal}$ is a downward directed collection of open subsets of $\spectrum(A)$ that witnesses condition (ii)  in Definition \ref{nontrivial real function} for $\widehat{A}$ and $\psi$.
\end{proof}

\begin{lemma}\label{approximate containment and integral}
Let $\Omega$ be a locally compact space, and let $A\subseteq \C(\Omega)$. Suppose that $A$ approximately contains a nontrivial real function $\phi$. Then:
\begin{enumerate}
	\item $\phi$ is $\mu$-measurable for every regular Borel measure $\mu$ on $\Omega$.
	\item For each $\nu\in\Measures(\Omega)$ with a compact support, there is $(h_n)_{n=1}^\infty\subseteq A$ such that $\int h_n\dd\nu\to\int \phi\dd\nu$.
\end{enumerate}
\end{lemma}
\begin{proof} (i) Take a compact subset $K$ of $\Omega$. By condition (i) of Definition \ref{nontrivial real function} for $A$ and $\phi$, we see that $K$ has a compact subset $K_0$ such that $\phi$ is Borel on $K_0\cup L$ for every compact subset $L$ of $K\setminus K_0$. Since $\mu$ is regular, we can find a sequence $(L_n)$ of compact subsets of $K\setminus K_0$ such that $\mu\left(K\setminus K_0\setminus\bigcup_{n=1}^\infty L_n\right)=0$. Since $\phi$ is Borel on $K_0\cup\bigcup_{n=1}^\infty L_n$, it follows that $\phi$ is $\mu$-measurable on $K$. Statement (i) then follows from standard measure theory (see, for example, \cite[Proposition 7.5.1]{Cohn}).

(ii) Let $\nu$ be a complex regular Borel measure on $\Omega$ with support $K$ being compact. Note that because $\phi$ is $\abs{\nu}$-measurable, the integral $\int\phi\dd\nu$ makes sense. Let $K_0$ and $\alpha_K$ be as specified in condition (i) of Definition \ref{nontrivial real function} for $A$, $\phi$, and $K$. Take $\varepsilon>0$ arbitrary. The regularity of $\nu$ gives us a compact subset $L\subseteq K\setminus K_0$ such that $\abs{\nu}(K\setminus K_0\setminus L)<\varepsilon/(\alpha_K+\abs{\phi}_\Omega+1)$, where $\abs{\phi}_\Omega$ is the uniform norm of $\phi$ on $\Omega$. By condition (i) of Definition \ref{nontrivial real function} again, there is a sequence $(f_n)_{n=1}^\infty\subseteq A$ that converges pointwise to $\phi$ on $K_0\cup L$ and each $\abs{f_n}\le \alpha_K$ on $K$. Thus we see that
\begin{align*}
	\abs{\int f_n\dd\nu-\int\phi\dd\nu}\le\int_{K_0\cup L}\abs{f_n-\phi}\dd\abs{\nu}+(\alpha_K+\abs{\phi}_\Omega)\frac{\varepsilon}{\alpha_K+\abs{\phi}_\Omega+1}\to \frac{(\alpha_K+\abs{\phi}_\Omega)\varepsilon}{\alpha_K+\abs{\phi}_\Omega+1}\,, 
\end{align*} 
by the bounded convergence theorem. Hence, there is an $h\in A$ such that $\abs{\int h\dd\nu-\int \phi\dd\nu}<\varepsilon$, and so the result follows.
\end{proof}

\begin{definition}\label{spectrum in a concrete representation}
Let $A$ be a Banach algebra that is at the same time a subalgebra of $\C(\Omega)$ for some topological space $\Omega$. Then we write  $\spectrum(A)=\Omega$ if the natural mapping $\Omega\to\spectrum(A)$ is a homeomorphism. 
\end{definition}

\begin{lemma}\label{spectrum in a concrete representation 2}
Let $A$ be a Banach algebra that is also a subalgebra of $\C(\Omega)$ for some topological space $\Omega$. Then the following are equivalent
\begin{enumerate}
	\item $\spectrum(A)=\Omega$. 
	\item $A\subseteq \C_0(\Omega)$ and the natural mapping $\Omega\to\spectrum(A)$ is a bijection.
\end{enumerate} 
\end{lemma}
\begin{proof}The implication (i)$\Rightarrow$(ii) is obvious. The implication (ii)$\Rightarrow$(i) follows from \cite{Whyburn65} as in the proof of Lemma \ref{real element vs real function}; note that this implies in particular that $\Omega$ must be a locally compact space. 
\end{proof}

Following Rieffel \cite{Rieffel} and Walter \cite{Walter74}, we define:

\begin{definition}\label{Tauberian algebra}
Let $A$ be a Banach algebra and let $\Omega$ be a topological space. Then $A$ is \emph{a Tauberian subalgebra of $\C(\Omega)$} or \emph{a Tauberian algebra on $\Omega$} if $A$ is a subalgebra of $\C(\Omega)$ and $A\cap \C_c(\Omega)$ is dense in $A$. Note that a Tauberian algebra on $\Omega$ must necessarily a subalgebra of $\C_0(\Omega)$.

A commutative Banach algebra $A$ is \emph{Tauberian} if $\widehat{A}$ is a Tauberian algebra on $\spectrum(A)$.
\end{definition}

For example, the Fourier algebra $\Fa(G)$ on any locally compact group $G$ is a Tauberian algebra (on $G$). This result was proved by Wiener for $G=\reals$, by Godement and Segal, independently, for abelian $G$ (see \cite[p. 170]{RS} for details\footnote{In \cite[Definition 2.2.9]{RS}, Reiter and Stegeman call  an algebra that satisfies the assumption in our Definition \ref{Tauberian algebra} \emph{in addition to} some other conditions a \emph{Wiener algebra}.}), and by Eymard  for general $G$ \cite[Corollaire 3.38]{Eymard}.

We also need another concept with a harmonic analysis origin. Following Rieffel \cite[Definition 7.1]{Rieffel}, we define:

\begin{definition}\label{Eberlein algebra}
Let $A$ be a Banach algebra, and let $\Omega$ be a topological space. Suppose that $A$ is also a subalgebra of $\C(\Omega)$. An \emph{$\Omega$-Eberlein function for $A$} is a continuous function $\phi$ on $\Omega$ with a property that there is a constant $k$ such that $\abs{\sum_{i=1}^m \alpha_i \phi(s_i)}\le k$ whenever $\alpha_i\in\complexs$ and $s_i\in \Omega$ with
\[	
			\abs{\sum_{i=1}^m\alpha_i f(s_i)}\le \norm{f}\quad(f\in A)\,.
\]
We say that $A$ is \emph{an Eberlein subalgebra of $\C(\Omega)$} or \emph{an Eberlein algebra on $\Omega$} if in addition it contains every  $\Omega$-Eberlein function for itself. 
\end{definition}

A similar, but different, concept was introduced by Takahasi and Hatori in \cite{TH}: They call a commutative Banach algebra $A$ a \emph{BSE-algebra} if the $\spectrum(A)$-Eberlein functions for $\widehat{A}$ (given the quotient norm induced from $A$) \footnote{$\spectrum(A)$-Eberlein functions for $\widehat{A}$ is called \emph{BSE-functions} (for $A$) in \cite{KU}.} are precisely the restriction to $\spectrum(A)$ of the Gelfand transforms of elements of the multiplier algebra $\Mcal(A)$ of $A$. An important difference between this and Definition \ref{Eberlein algebra} is that, in the definition of BSE-algebras, the functions have $\spectrum(A)$ as their domain instead of some given topological space $\Omega$ for which $A$ could be regarded as a subalgebra of $\C(\Omega)$, and so to compare a function with an element of $A$ one needs to evaluate them at all of $\spectrum(A)$ not just at points of $\Omega$. This is one of the reasons why the classes of Eberlein algebras and of BSE-algebras could be quite different as discussed below.\footnote{We remark that a completely different concept with a similar name, that of \emph{the Eberlein algebra of} a locally compact group $G$, is defined by Elgun \cite{Elgun} as the uniform closure of $\FSa(G)$.}

Both of the concepts of Eberlein functions/algebras and of BSE-algebras have origin in the Bochner--Schoenberg--Eberlein theorem, proved by Bochner and Schoenberg for $\reals$ and by Eberlein for all locally compact abelian groups. We give the statement of this theorem in the Fourier algebra setting for the ease of discussion (see \cite[p. 32]{Rudin}).

\begin{thm}[Bochner--Schoenberg--Eberlein]
Let $G$ be a locally compact abelian group, and let $\phi\in \C(G)$. Then the following are equivalent:
\begin{enumerate}
	\item $\phi\in \FSa(G)$.
	\item there is a constant $k$ such that 
$\abs{\sum_{i=1}^m \alpha_i \phi(s_i)}\le k\cdot \norm{\sum_{i=1}^m\alpha_i s_i}_{\VN(G)}$ for every $\alpha_i\in\complexs$ and $s_i\in G$.
\end{enumerate}
Here we interpret $G$ as left-translation operators on $\Ltwo(G)$, a subset of $\VN(G)=\Linfty(\widehat{G})$, which is the dual of $\Fa(G)=\Lone(\widehat{G})$, in a natural manner. 
\end{thm}

The generalisation of the above statement of  the Bochner--Schoenberg--Eberlein theorem (BSE theorem for short) to a general locally compact group $G$ was shown by Kaniuth and \"{U}lger to be true if and only if $G$ is amenable \cite[Theorem 5.1]{KU}; in particular,  $\Fa(G)$ is a BSE-algebra if and only if $G$ is amenable. They also showed that, for nilpotent $G$, $\FSa(G)$ is a BSE-algebra if and only if $G$ is compact, but that there are non-nilpotent non-compact $G$ for which $\FSa(G)$ is also a BSE-algebra. 
 
On the other hand, noting that actually $G\subset AP(\widehat{G})=\C(\widehat{G_d})=\C^*(G_d)$ when $G$ is abelian, and so in the BSE theorem, we could replace  $\norm{\sum_{i=1}^m\alpha_i s_i}_{\VN(G)}$ by $\norm{\sum_{i=1}^m\alpha_i s_i}_{\C^*(G_d)}$, where $G_d$ is the group $G$ with the discrete topology. In this different interpretation, the BSE theorem has a generalisation to all $G$ given in \cite[Corollaire 2.24]{Eymard}. In particular, this shows that the Fourier--Stieltjes algebra $\FSa(G)$ is always an Eberlein subalgebra of $\C(G)$, for any $G$. The class of closed translation-invariant Eberlein subalgebras of $\FSa(G)$ will be described in \S\ref{Eberlein subalgebras of Fourier--Stieltjes algebras}.

To simplify our terminologies, we shall follow the next convention:

\begin{convention}
Let $\Omega$ be a topological space, and let $B$ be a subalgebra of $\C(\Omega)$; the association between $B$ and $\Omega$ should be clear from the context. We shall call $A$ \emph{a Tauberian/Eberlein subalgebra of $B$} if it is a Tauberian/Eberlein algebra on $\Omega$ that is contained in $B$. 
\end{convention}

\section{Invariant subalgebras of Fourier--Stieltjes algebras}
\label{Eberlein subalgebras of Fourier--Stieltjes algebras}
\label{Invariant subalgebras of Fourier algebras}

\noindent In this section, we shall prove some auxiliary results about translation-invariants subalgebras of the Fourier--Stieltjes algebra $\FSa(G)$ of a locally compact group $G$; we believe that some of these results are of independent interest.

First, we shall characterise when a closed subalgebra of $\Fa(G)$ is the whole of $\Fa(G)$. A starting point for us is the following immediate consequence of \cite[Theorem 2.1]{BLSchl} or of \cite[Theorem 9]{TT}.

\begin{theorem}
Let $G$ be a locally compact group, and let $A$ be a closed translation-invariant subalgebra of $\Fa(G)$ that is   conjugation-closed and separates points of $G$. Then $A=\Fa(G)$.
\end{theorem}

Since complex-conjugation is actually the natural involution for $\Fa(G)$, the above result characterises $\Fa(G)$ among certain $*$-subalgebras of itself.

Without the conjugation-closedness assumption, the above result fails: take $G:=\unitcircle$ and let $A$ be the closure of the set of functions in $\Fa(\unitcircle)$ associated with complex polynomials of one variable. Then $A$ is a proper closed translation-invariant subalgebra of $\Fa(\unitcircle)$ that separates points of $\unitcircle$. Thus we would need to replace the conjugation-closedness assumption by some other conditions if we want the conclusion above still holds.

When $G$ is a locally compact \emph{abelian} group, it was proved by Rieffel in \cite[\S 6]{Rieffel} that if $A$ is a closed translation-invariant Tauberian subalgebra of $\Fa(G)$ such that $\spectrum(A)=G$ (in the sense of Definition \ref{spectrum in a concrete representation}, but see also Lemma \ref{spectrum in a concrete representation 2}), then $A=\Fa(G)$; this precise statement is not presented as such in \cite{Rieffel} but could be easily extracted, and, for convenience, it is also translated to the Fourier algebra formulation here. In fact, this was strengthened by Taylor in \cite{Taylor69}, using his deep work on Laplace transform in \cite{Taylor68}: the result, when again formulated in the Fourier algebra language, states that if $A$ is a closed translation-invariant subalgebra of $\FSa(G)$ such that $\spectrum(A)=G$, then $\Fa(G)\subseteq A\subseteq \FSa_0(G)$, where 
\[
	\FSa_0(G):=\FSa(G)\cap \C_0(G)
\]
is \emph{the Rajchman algebra of $G$} \footnote{This algebra, for an abelian locally compact group, has been studied for a long time; see, for example, the book \cite{GM} -- a standard reference in commutative harmonic analysis. A recent study, for the nonabelian case, was carried out by Ghandehari in \cite{Ghandehari}.}; remark here that a Tauberian subalgebra of $\FSa(G)$ is necessarily a subalgebra of $\Fa(G)$. The proofs of both of these results rely heavily on the structure theory of locally compact abelian groups, applied to the dual group of the group $G$, and thus are not amenable to generalisation to nonabelian situation. Here, we shall not be able to generalise these theorems completely, but we shall prove some partial results that extend them to a certain class of locally compact (nonabelian) groups.

\begin{remark}
Remarks on the condition $\spectrum(A)=G$: 
\begin{enumerate}
\item For Taylor's type of problem, requiring merely that the natural mapping $G\to\spectrum(A)$ is bijective is too weak an assumption: Take any nondiscrete locally compact group $\Gamma$, set $G:=\Gamma_d$ and  $A:=\Fa(\Gamma)$. Then $A$ is a closed translation invariant subalgebra of $\FSa(G)$ and the natural mapping $G\to\spectrum(A)$ is bijective, but $A$ neither contains $\Fa(G)$ nor is contained in $\FSa_0(G)$. 
\item On the other hand, for Rieffel's type of problem (even without the Tauberian assumption), requiring that the natural mapping $G\to\spectrum(A)$ is bijective would be sufficient to imply the stronger condition that $\spectrum(A)=G$ as shown by Lemma \ref{spectrum in a concrete representation 2}.
\end{enumerate}
\end{remark}

In the following series of lemmas, let $G$ be a locally compact group, and let $A$ be a closed translation-invariant Tauberian subalgebra of $\Fa(G)$. 

\begin{lemma}\label{when A=A(G), compact-convergence dense case}
If $A$ is dense in $\C_0(G)$ in the compact-convergence topology on $G$, then $A=\Fa(G)$. 
\end{lemma}
\begin{proof}
For each function $f$ on $G$, let us denote by $\check{f}$ the function $t\mapsto f(t^{-1})$ on $G$. Set 
\[
	\check{A}:=\set{\check{f}\colon f\in A}\,.
\]
Take $\phi\in\C_c(G)$ and $f\in A\cap \C_c(G)$. The assumption implies that $\phi \check{f}$ is the uniform limit of functions in $\check{A}\check{f}$. Hence, $\phi \check{f}$ is the limit in $\Ltwo(G)$ of functions in $\check{A}\cap \C_c(G)$. By a similar argument and the Tauberian property of $A$, we obtain that $\phi$ is the limit in $\Ltwo(G)$ of functions in $\check{A}\cap \C_c(G)$, and so $\check{A}\cap \C_c(G)$ is dense in $\Ltwo(G)$. The result then follows from an argument in \cite[p. 157]{Walter74} that we repeat here briefly for the reader's convenience: for each $g\in A$, the mapping $t\mapsto L_tg, G\to A,$ is continuous, and so $f*g$ defines an element in $A$ whenever $f\in \C_c(G)$ and $g\in A\cap \C_c(G)$. This, the density of $\check{A}\cap \C_c(G)$ in $\Ltwo(G)$, and the fact \cite{Eymard} that $\Fa(G)=\set{f*\check{g}\colon f,g\in \Ltwo(G)}$, then complete the argument.
\end{proof}

\begin{lemma}\label{when A=A(G), compact case}
If $G$ is compact and $\spectrum(A)=G$, then $A=\Fa(G)$.
\end{lemma}
\begin{proof}
Denote by $\Scal$ the union of the supports in $\widehat{G}$ of the continuous unitary representations associated with positive definite functions in $A$. Since $A$ is a closed translation-invariant subalgebra of $\Fa(G)$, $\Scal$ is a sub-semidual of $\widehat{G}$; noting also that since $G$ is compact, $\widehat{G}$ is discrete (see, for example, \cite[Proposition 1.70]{KT}). Set 
\[
	\Scal^*:=\set{\conjugate{\pi}\colon\ \pi\in\Scal}\quad\textrm{and}\quad \Scal_1:=\Scal\cap \Scal^*\,.
\]
Then $\Scal_1$ is a subdual of $\widehat{G}$. Note that the trivial representation $1_G$ belongs to $\Scal_1$ for $A$ is necessarily unital by \v{S}hilov's idempotent theorem (see \cite[Theorem 2.4.33]{Dales}).  We \emph{claim} that if $\pi\in\Scal\setminus\Scal_1$ and $\rho\in\Scal$, then the support of $\pi\otimes\rho$ must be contained in $\Scal\setminus\Scal_1$. Indeed, assume towards a contradiction that $\mu\preccurlyeq \pi\otimes\rho$ for some $\mu\in\Scal_1$. Then by applying both directions of a theorem of Fell \cite[Theorem 4]{Fell} we obtain
\[
	1_G\preccurlyeq \mu\otimes\conjugate{\mu}\preccurlyeq  \pi\otimes\rho\otimes\conjugate{\mu}\,,
\]
and then $\conjugate{\pi}\preccurlyeq \rho\otimes\conjugate{\mu}$. This implies that $\conjugate{\pi}\in\Scal$, and so $\pi\in\Scal_1$; a contradiction to the assumption that $\pi\in \Scal\setminus\Scal_1$.

Set $B:=\FSa_{\Scal_1}(G)$. Then $B$ is a closed translation-invariant subalgebra of $\Fa(G)$; indeed, we have
\[
	B=\set{f\in\Fa(G)\colon\ f\ \textrm{is constant on cosets of}\ K}\,,
\]
where $K$ is the common kernel of $\Scal_1$. Denote by $z$ and $z_K$ the support projections of $A$ and $B$ in $\VN(G)=\dual{\Fa(G)}$, respectively. Take $f\in A\cap P(G)$. Then $f\cdot (1-z_K)$ is also a positive definite function in $A$, whose associated continuous unitary representation of $G$ is supported in $\Scal\setminus\Scal_1$. It follows from a remark above that the continuous unitary representation of $G$ associated with $(f\cdot (1-z_K))g$ is supported in $\Scal\setminus \Scal_1$ for every $g\in A\cap P(G)$. Thus
\[
	[(f\cdot (1-z_K))g]\cdot z_K=0
\]
for every $f,g\in A\cap P(G)$, and hence, by linearity, for every $f,g\in A$. It follows that
\[
	(fg)\cdot z_K=[(f\cdot z_K)(g\cdot z_K)]\cdot z_K=(f\cdot z_K)(g\cdot z_K) 
\]
for every $f,g\in A$; the second equality is because $B$ is a subalgebra of $\Fa(G)$. Since the constant function $\tuple{1}\in A\cap B$, we see that $z_Kz\neq 0$. Thus it follows that $z_Kz\in z\VN(G)=\dual{A}$ is a character of $A$. The hypothesis then implies that $z_Kz=z\lambda_t$ for some $t\in G$. This can only happen if $z_Kz=z$, and so $A\subseteq B$. In particular, we must have $K=\set{1}$ and $B=\Fa(G)$. Now, by \cite[Proposition 1.21]{Eymard}, every function in $B=\Fa(G)$ is a uniform limit of functions in $A$. Hence, $A$ is uniformly dense in $\C(G)$. The result then follows from the previous lemma.
\end{proof}

\begin{remark}
After the paper has been submitted, we discovered that Lemma \ref{when A=A(G), compact case} follows immediately from the main theorem of \cite{Izzo} and Lemma \ref{when A=A(G), compact-convergence dense case}. We are grateful to Alexander Izzo for explaining his result to us. The paper \cite{Izzo} actually deals with a more general situation, and so its proof is deeper than the simple argument above. We decide to keep the above proof for completeness. 
\end{remark}

\begin{theorem}\label{when A=A(G), nonzero real function case}
Let $G$ be a locally compact group, and let $A$ be a closed translation-invariant Tauberian subalgebra of $\Fa(G)$. Suppose that $\spectrum(A)=G$ and that $A$ approximately contains a nontrivial real function. Then $A=\Fa(G)$.
\end{theorem}
\begin{proof}
Our idea here is based on the proof of Bishop's generalisation of the Stone--Weierstrass theorem \cite{Bishop} (cf. also \cite[Theorem 13.1]{Rudin2}).

Let $\phi$ be the bounded real function specified in Definition \ref{nontrivial real function} for $A$. Then $\phi$ is measurable with respect to the Haar measure on $G$ by Lemma \ref{approximate containment and integral}(i), and so defines an element of $\Linfty(G)$. Set
\[
	B:=\set{L_s(\phi*f)=(L_s\phi)*f\colon f\in \C_c(G)\,, s\in G}\,.
\]
Then $B\subseteq \C^b(G)$ (cf. \cite[Proposition 2.39]{Folland}). Denote by $\Fcal$ the collection of subsets $F$ of $G$ that are contained in some level set of $\psi$ for every $\psi\in B$. It is easy to see that $\closure{\bigcup_j F_j}\in \Fcal$ whenever $F_j\in \Fcal$ with $\bigcap_j F_j\neq\emptyset$. Also, $\Fcal$ is translation-invariant. It follows that the maximal sets in $\Fcal$, which must exist, are the cosets of some closed normal subgroup $F$ of $G$. We \emph{claim} that $F$ is compact. Indeed, take $s\in F$. The definition of $F$ implies that
\[
	(L_s\phi)*f=\phi*f\qquad(f\in \C_c(G))\,,
\]
and so $L_s\phi=\phi$ as elements of $\Linfty(G)$. Condition (ii) in Definition \ref{nontrivial real function} then implies that $U \setminus s^{-1} \closure{U}$ is locally null with respect to the Haar measure on $G$, for each $U$ belonging to the collection $\Ucal$ associated with $\phi$. Since each such $U$ is open, we see that $\closure{U}\subseteq s^{-1}\closure{U}$ for every $U\in \Ucal$. So $s L\subseteq L$, where $L:=\bigcap_{U\in\Ucal}\closure{U}$ is a nonempty compact set by condition (ii) of Definition \ref{nontrivial real function}. Thus $F\subseteq LL^{-1}$, and so must be compact.

Consider $A|_{F}:=\set{f|_{F}\colon f\in A}$. Then $A|_{F}$ is a translation-invariant subalgebra of $\Fa(F)$. We \emph{claim} that $A|_{F}$ is closed in $\Fa(F)$. Indeed, certainly $A|_{F}$ is a Banach algebra with the quotient norm $\norm{\cdot}_{A|_{F}}$ induced from that of $A$. By Herz's restriction theorem \cite{Herz}, the restriction $f\mapsto f|_{F}, \Fa(G)\to \Fa(F),$ is a quotient map whose dual gives a natural $*$-isomorphism from $\VN(F)$ onto the \wstar-subalgebra $W$ of $\VN(G)$ generated by $\set{\lambda_G(t)\colon t\in F}$, mapping $\lambda_{F}(t)\mapsto \lambda_G(t)$. Take $f\in A$. Then $f=f\cdot z$, where $z$ is the support projection of $A$ in $\VN(G)$, and so 
\begin{align*}
	\norm{f|_{F}}_{\Fa(F)}&=\sup\set{\abs{\duality{f}{x}}\colon\ x\in \closedball{W}{1}}\\
	&=\sup\set{\abs{\duality{f}{zx}}\colon\ x\in \closedball{W}{1}}\\
	&=\sup\set{\abs{\duality{f}{x}}\colon\ x\in \closedball{(zW)}{1}}=\norm{f|_{F}}_{A|_{F}}
\end{align*}
where the third equality is because $x\mapsto zx$ is a $*$-epimorphism from a \wstar-algebra $W$ onto necessarily another one $zW$. This gives the claim. Since $\spectrum(A)=G$, it is easy to see that $\spectrum(A|_{F})=F$. Lemma \ref{when A=A(G), compact case} then can be applied to show that $A|_{F}=\Fa(F)$.

Now, consider a fixed compact subset $K$ of $G$ and set
\[
	M:=\set{\mu\in\Measures(K)=\dual{\C(K)}\colon \ \int_K f\dd\mu=0\quad (f\in A)}\,.
\]
Take $\mu\in M$. We \emph{claim} that $\psi\mu\in M$ for any $\psi\in B$.  Indeed, says $\psi=L_s(\phi*f)$ for some $s\in G$ and $f\in\C_c(G)$. Take $g\in A$. We see, for every bounded Haar measurable function $h$ on $G$, that
\begin{align*}
	\int_KgL_s(h*f)\dd\mu &=\int_K\int_Gg(t)h(u)f(u^{-1}st)\dd u\dd{\mu}(t)\\
	&=\int_Gh(u)\left(\int_Kg(t)f(u^{-1}st)\dd{\mu}(t)\right)\dd u=:\int h\dd\nu\,,
\end{align*}
where $\nu\in\Measures(G)$ has a compact support. Now, for each $h\in A$, the right-translation $t\mapsto R_th, G\to A,$ is continuous, and so $h*f\in A$ since $f\in \C_c(G)$. It follows that $gL_s(h*f)\in A$, and hence 
\[
	\int h\dd\nu=\int_KgL_s(h*f)\dd\mu=0\qquad(h\in A)\,,
\]
since $\mu\in M$. This and Lemma \ref{approximate containment and integral}(ii) then imply that 
\[
	\int_Kg\psi\dd\mu=\int_KgL_s(\phi*f)\dd\mu=\int \phi\dd\nu=0\,.
\]
This proves the claim.

Next, assume towards a contradiction that $M\neq \set{0}$. Then, by the Kre\u{\i}n--Milman theorem, an extreme point $\mu$ of $\closedball{M}{1}$ exists with $\norm{\mu}=1$. We \emph{claim} that $\mu$ is supported at a singleton. Indeed, denote by $S$ the support of $\mu$; translating the whole setting if necessary, we assume that $1\in S$. Take any $\psi\in B$. By scaling if necessary, we may and shall assume that $-1<\psi<1$ on $K$. Consider two measures
\[
	\mu_1:=\frac{1}{2}(1+\psi)\mu\quad\textrm{and}\quad \mu_2:=\frac{1}{2}(1-\psi)\mu\,.
\]
Then, by the previous claim, we see that $\mu_1,\mu_2\in M$ and 
\[
	\norm{\mu_1}+\norm{\mu_2}=\int_K\frac{1}{2}(1+\psi)\dd\abs{\mu}+\int_K\frac{1}{2}(1-\psi)\dd\abs{\mu}=\abs{\mu}(K)=\norm{\mu}=1\,,
\]
while $\mu=\mu_1+\mu_2$. The extremal property of $\mu$ then implies that $\mu_1$ is a scalar multiplication of $\mu$. Thus $\psi$ is constant $\abs{\mu}$-a.e. on $K$. In other words, since $\psi$ is continuos, the support $S$ of $\mu$ must be contained in a level set of $\psi$. Hence $S\in \Fcal$, and so $S\subseteq F$. On the other hand, the same argument shows that any real function in $A|_S$ must be constant. But since $A|_F=\Fa(F)$, real functions in $A|_S$ are constant if and only if $S$ is a singleton. Says $S=\set{s}$. This then means that $f(s)=0$ for every $f\in A$; a contradiction. 

Thus we have shown that $A|_K$ is uniformly dense in $\C(K)$ for every compact subset $K$ of $G$. In particular, $A$ is compact-convergence dense in $\C_0(G)$. The result then follows by Lemma \ref{when A=A(G), compact-convergence dense case}.
\end{proof}

We have the following extension of Taylor's result \cite[Theorem 1]{Taylor69}. 

\begin{corollary}\label{when A=A(G), compact connected component case}
Let $G$ be a locally compact group whose connected component of the identity is compact, and let $A$ be a closed translation-invariant subalgebra of $\FSa(G)$ with $\spectrum(A)=G$. Then 
\[
	\Fa(G)\subseteq A\subseteq \FSa_0(G)\,.
\]
\end{corollary}
\begin{proof}
Since $G$ must contain an open compact group $K$, \v{S}hilov's idempotent theorem again implies that $\chi_K\in A$, and so Theorem \ref{when A=A(G), nonzero real function case} can be applied if $A$ is assumed to be Tauberian. In the general case, as proved in that theorem, we could still conclude that $A$ is compact-convergence dense in $\C_0(G)$. Using $\chi_U$ for the unions $U$ of finitely many cosets of $K$, we obtain that $A\cap \C_c(G)$ is compact-convergence dense in $\C_0(G)$. The argument of the proof of Lemma \ref{when A=A(G), compact-convergence dense case} then shows that $\Fa(G)\subseteq A$. The conclusion that $A\subseteq \FSa_0(G)$ is obvious.
\end{proof}

Next, we shall give a detailed description of the class of closed translation-invariant Eberlein subalgebras of $\FSa(G)$. This will be needed in  \S \ref{Fourier--Stieltjes algebras}.

\begin{proposition}\label{Eberlein is weak*-closed}
Let $G$ be a locally compact group, and let $B$ be a closed translation-invariant Eberlein subalgebra of $\FSa(G)$. Then $B$ is weak$^*$-closed in $\FSa(G)$, and so $B=\FSa_\Scal(G)$ for some sub-semidual $\Scal$ of $G$.
\end{proposition}
\begin{proof}
It is sufficient by \cite[Proposition 1.21]{Eymard} to show that $B\cap P(G)$ is closed in $P(G)$ under the compact--convergence topology on $G$. In fact,  $B\cap P(G)$ is closed in $\C(G)$ under the pointwise--convergence topology on $G$ (but $B$ itself is not necessarily closed under that topology). Indeed, suppose that $(f_\alpha)$ is a net in $B\cap P(G)$ that converges to some $f\in\C(G)$ under the pointwise--convergence topology on $G$. Then $f$ must be a continuous positive definite function, i.e. $f\in P(G)$. To prove that $f$ is an $G$-Eberlein function for $B$, take  $\alpha_i\in\complexs$ and $s_i\in \Omega$ with
\[	
			\abs{\sum_{i=1}^m\alpha_i g(s_i)}\le \norm{g}\quad(g\in B)\,.
\]
Then since $(f_\alpha)\subset B$ and converges pointwise to $f$, we obtain
\begin{align*}
	\abs{\sum_{i=1}^m\alpha_i f(s_i)}&=\lim_\alpha \abs{\sum_{i=1}^m\alpha_i f_\alpha(s_i)}\le \lim_\alpha \norm{f_\alpha}=\lim_\alpha f_\alpha(1)=f(1)=\norm{f}_{\FSa(G)}\,.
\end{align*}
Thus $f$ is an $G$-Eberlein function for $B$, and hence, $f\in B$.
\end{proof}

To describe closed translation-invariant Eberlein subalgebras of $\FSa(G)$ completely, we have to actually consider sub-semiduals of $G_d$ instead of those of $G$. Recall that $G_d$ is the group $G$ with the discrete topology. 

\begin{definition}
Let $G$ be a locally compact group. For each sub-semidual $\Scal$ of $G_d$, we define $\FSa_\Scal^c(G):=\FSa_\Scal(G_d)\cap \C(G)$. A sub-semidual $\Scal$ of $G_d$ is called \emph{thick} if $\FSa_\Scal^c(G)$ separates points of $G$.
\end{definition}

Note that $\FSa_\Scal^c(G)$ is a closed subalgebra of $\FSa(G)$, since the uniform norm on $G$ is majorized by the norm of $\FSa(G_d)$.

\begin{lemma}\label{Eberlein cover}
Let $G$ be a locally compact group, and let $A$ be a closed translation-invariant subalgebra of $\FSa(G)$. Denote by $\Scal$ the closure in $\widehat{G_d}$ of the union of the supports, also in $\widehat{G_d}$, of the unitary representations of $G$ associated with positive definite functions in $A$. Then $\Scal$ is a sub-semidual of $G_d$ and 
\[
	\set{\phi\in\C(G)\colon\ \phi\ \textrm{is a $G$-Eberlein function for}\ A}=\FSa_\Scal^c(G)
\]
is an Eberlein subalgebra of $\FSa(G)$.
\end{lemma}
\begin{proof}
First of all, since $A$ is a closed translation-invariant subalgebra of $\FSa(G)$, and hence of $\FSa(G_d)$, it follows that $\Scal$ is a sub-semidual of $G_d$. Denote by $B$ the left-hand side of the equation to be shown above. Then $A\subseteq B\cap \FSa_\Scal^c(G)$.

Take $f\in \FSa_\Scal(G_d)\cap \C(G)$. Then $f\in \FSa(G)$. Write $f=u_1-u_2+\init (u_3-u_4)$, by first breaking it into the real and imaginary parts as a bounded linear functional on $\C^*_\Scal(G_d)$, and then using the Jordan decomposition of self-adjoint linear functionals on that \cstar-algebra. We see that this decomposition is precisely the same one when we view $f\in \FSa(G_d)$ as a bounded linear functional on $\C^*(G_d)$ (see \cite[Remark 2.6(2)]{Eymard}). But then, it follows from \cite[Theorem 2.20(2)]{Eymard} that each $u_k\in P(G)$;  in particular, $u_k\in \C(G)$. However, since $u_k\in P_\Scal(G_d)$, it is the limit of a net in $A\cap P(G)$ in the compact-convergence (i.e. pointwise-convergence) topology of $G_d$. It follows as in the previous proposition then that each $u_k\in B$. Hence $f\in B$.

It remains to prove that if $\phi$ is a $G$-Eberlein function for $\FSa_\Scal^c(G)$, then $\phi\in \FSa_\Scal(G_d)$.  Consider the natural map $\Phi:\C^*_\Scal(G_d)\to \dual{\FSa_\Scal^c(G)}$, defined by dualities with the inclusion $\FSa_\Scal^c(G)\subseteq \FSa_\Scal(G_d)$. Note that $\dual{\FSa_\Scal^c(G)}$ is a \wstar-algebra (in fact, a natural quotient of $W^*(G):=\bidual{\C^*(G)}$, the universal von Neumann algebra of $G$ (see \cite{Walter72}), by a normal ideal) since $\FSa_\Scal^c(G)$ is a closed translation-invariant subalgebra of $\FSa(G)$. Moreover, the natural image of $G$ in $\sigma(\FSa_\Scal^c(G))\subseteq\dual{\FSa_\Scal^c(G)}$ is a subgroup of $\Ucal(\dual{\FSa_\Scal^c(G)})$, the unitary group of $\dual{\FSa_\Scal^c(G)}$. Thus $\Phi:\C^*_\Scal(G_d)\to \dual{\FSa_\Scal^c(G)}$ is a $*$-homomorphism. Since $\FSa_\Scal^c(G)$ is weak$^*$-dense in $\FSa_\Scal(G_d)$ (because $A$ is), $\Phi$ is injective, and hence, isometric. This and the discreteness of the topology of $G_d$ imply that $\phi$, being a $G$-Eberlein function for $\FSa_\Scal^c(G)$, defines an element of $\dual{\C^*_\Scal(G_d)}=\FSa_\Scal(G_d)$. Hence $\phi\in \FSa_\Scal(G_d)\cap \C(G)=\FSa_\Scal^c(G)$.  This shows that $\FSa_\Scal^c(G)$ is an Eberlein algebra on $G$. Moreover, if $\phi\in B$, then certainly $\phi$ is $G$-Eberlein for $B_S^c(G)$ since the latter contains A, and so $\phi$ belongs to $\FSa_\Scal^c(G)$. This also shows that $B\subseteq\FSa_\Scal^c(G)$.
\end{proof}

Consequently, we have the following description of closed translation-invariant Eberlein subalgebras of $\FSa(G)$.

\begin{theorem}\label{when translation-invariant subalgebra is Eberlein}
Let $G$ be a locally compact group, and let $B$ be a closed translation-invariant subalgebra of $\FSa(G)$. Then the following are equivalent:
\begin{enumerate}
	\item $B$ is Eberlein. 
	\item $B=\FSa_\Scal^c(G)$ for  some sub-semidual $\Scal$ of $G_d$.
\end{enumerate}
If moreover $B$ is conjugation-closed, then $\Scal$ can be chosen to be a subdual of $G_d$.
\end{theorem}
\begin{proof}
(i)$\Rightarrow$(ii): This is immediate from Lemma \ref{Eberlein cover}.

(ii)$\Rightarrow$(i):  
Under the assumption of (ii), $B$ is a closed translation-invariant subalgebra of $\FSa(G)$. Denote by $\Tcal$ the closure in $\widehat{G_d}$ of the union of the supports, also in $\widehat{G_d}$, of the unitary representations of $G$ associated with positive definite functions in $B$. Then obviously $\Tcal\subseteq\Scal$. Thus we see that  $B\subseteq\FSa_\Tcal^c(G)\subseteq\FSa_\Scal^c(G)=B$. The result then follows again from Lemma \ref{Eberlein cover}, applied to $\Tcal$.
\end{proof}

It is rather straight forward to determine when a subdual of $G_d$ is thick:

\begin{proposition}\label{when a subdual is thick}
Let $G$ be a locally compact group, and let $\Scal$ be a subdual of $G_d$. Then the following are equivalent:
\begin{enumerate}
	\item $\Scal$ is thick.
	\item $\lambda_G\preccurlyeq \Scal$ {\rm(}as representations of $G_d${\rm)}.
	\item $\FSa_r(G)\subseteq \FSa_\Scal^c(G)$.
\end{enumerate}
\end{proposition}
\begin{proof}
(i)$\Rightarrow$(ii): It is sufficient to show that $\Fa(G)\subseteq \FSa_\Scal^c(G)$; this is because $\Fa(G)\cap P(G)$ is dense in $\FSa_r(G)\cap P(G)$ in the compact-convergence topology of $G$. But this follows from \cite[Theorem 1.3]{BLSchl}; $\FSa_\Scal^c(G)$ is weak$^*$-closed in $\FSa(G)$ by Theorem \ref{when translation-invariant subalgebra is Eberlein} and Proposition \ref{Eberlein is weak*-closed}. 

The implication (ii)$\Leftrightarrow$(iii) is standard while (iii)$\Rightarrow$(i) is obvious. 
\end{proof}

For the question of when a sub-semidual of $G_d$ is thick, we could only give a partial answer. To help in determining this, we need some auxiliary results about discrete groups. But before that, we have the following.

\begin{lemma}\label{group Cstar-algebra wrt a sub-semidual}
Let $G$ be a locally compact group, and let $\Scal$ be a sub-semidual of $G$. Then $\C^*_\Scal(G)$ admits a comultiplication that is induced from the comultiplication of $\C^*(G)$ through the natural mapping $\pi:\C^*(G)\to \C^*_\Scal(G)$. 
\end{lemma}
\begin{proof}
The natural $*$-epimorphism $\pi:\C^*(G)\to \C^*_\Scal(G)$ could be considered as the quotient mapping by some closed ideal $I$. Denote by $\Delta:\C^*(G)\to \C^*(G)\otimes_{\min}\C^*(G)$ the natural comultiplication of $\C^*(G)$. Also, denote by $\Acal$ the dual of $\C^*(G)\otimes_{\min}\C^*(G)$, so that the (algebraic) tensor product $\FSa(G)\otimes \FSa(G)$ is weak$^*$-dense in $\Acal$. Note that
\[
	\dual{\Delta}:\Acal\to \FSa(G), f\otimes g\mapsto fg,\qquad(f,g\in \FSa(G))\,.
\]
Set $B:=\FSa_\Scal(G)$, and let $\Bcal$ be the weak$^*$-closure of $B\otimes B$ in $\Acal$. Since $B$ is a $\C^*(G)$-submodule of $\FSa(G)$, it is easy to see that $\Bcal$ is a $\C^*(G)\otimes_{\min}\C^*(G)$-submodule of $\Acal$. It follows that $\Bcal=\dual{(\predual{\Bcal})}$ for some \cstar-algebra $\predual{\Bcal}$ which is the quotient of  $\C^*(G)\otimes_{\min}\C^*(G)$ by a closed ideal $\Jcal$. Since obviously $I\otimes \C^*(G)+\C^*(G)\otimes I\subseteq \Jcal$, we obtain a natural mapping $j$ in the following commutative diagram
\begin{equation*}
\SelectTips{eu}{12}\xymatrix{\C^*(G)\otimes\C^*(G) \ar[r]^-{} \ar@{->}[d]_{\pi\otimes\pi} & \C^*(G)\otimes_{\min}\C^*(G)\ar@{->}[d]^{\textrm{natural quotient}}  \\ 
            \C^*_\Scal(G)\otimes\C^*_\Scal(G)\ar[r]^-{j}  & (\C^*(G)\otimes_{\min}\C^*(G))/\Jcal\,.}
\end{equation*} 
The mapping $j$ has dense range, and is easily seen to be injective (by dualizing with $B\otimes B\subseteq \Bcal$). Thus it follows that there must be a natural $*$-epimorphism
\[
	\rho:(\C^*(G)\otimes_{\min}\C^*(G))/\Jcal\to \C^*_\Scal(G)\otimes_{\min}\C^*_\Scal(G)\,,
\]
extending $j^{-1}$. Since $B$ is an algebra, $\dual{\Delta}$ maps $\Bcal$ into $B$, and so $\Delta$ induces a $*$-homomorphism 
\[
	\Phi:\C^*_\Scal(G)\to (\C^*(G)\otimes_{\min}\C^*(G))/\Jcal\,.
\]
Thus we obtain a commutative diagram
\begin{equation*}
\SelectTips{eu}{12}\xymatrix{\C^*(G)\ar[r]^-{\Delta} \ar@{->}[d]_{\pi} & \C^*(G)\otimes_{\min}\C^*(G)\ar@{->}[d]^{\pi\otimes\pi}  \\ 
            \C^*_\Scal(G)\ar[r]^-{\rho\circ\Phi}  & \C^*_\Scal(G)\otimes_{\min}\C^*_\Scal(G)\,.}
\end{equation*} 
From there, it is easy to see that $\rho\circ\Phi$ is the desired comultiplication on $\C^*_\Scal(G)$. 
\end{proof}

\begin{lemma}\label{weak*-closed invariant subalgebra of B(G) when G is discrete}
Let $G$ be any discrete group, and let $B$ be a weak$^*$-closed translation-invariant subalgebra of $\FSa(G)$ that separates points of $G$. Then $\FSa_r(G)\subseteq B\subseteq \FSa(G)$.
\end{lemma}
\begin{proof}
Since $B$ is a weak$^*$-closed and translation-invariant subalgebra of $\FSa(G)$, it must have the form $\FSa_\Scal(G)=\dual{\C^*_\Scal(G)}$ for some sub-semidual $\Scal$ of $G$ (cf. Lemma \ref{algebra of sub-semidual}). Continuing with the notation in the previous proof, with $G$ discrete, $(\C^*(G),\Delta)$ is a compact quantum group in the sense of Woronowicz \cite{Woronowicz}, i.e. $\C^*(G)$ is unital and the sets
\[
	\Delta(\C^*(G))(\C^*(G)\otimes 1) \quad\textrm{and}\quad (\C^*(G)\otimes 1)\Delta(\C^*(G))
\]
are dense in $\C^*(G)\otimes_{\min}\C^*(G)$. It can then be seen from the preceding commutative diagram that the same properties are hold by $(\C^*_\Scal(G),\rho\circ\Phi)$. It follows from \cite{Woronowicz} for the separable case and \cite{VanDaele} for the general case that there exists a Haar state of $\C^*_\Scal(G)$, i.e. a positive definite function $\omega\in B=\FSa_\Scal(G)$ with $\omega(1)=1$ such that $f\omega=f(1)\omega$ for every $f\in B$. Since $B$ separates points of $G$, we see that $\omega=\delta_1$, the point mass at the identity of $G$. The conclusion then follows.
\end{proof}

The following interesting result seems to be new (cf. \cite[Corollary 1.4]{BLSchl}). 

\begin{corollary}\label{weak*-closed invariant subalgebra of B(G) when G is discrete 2}
Let $G$ be any amenable discrete group. Then every weak$^*$-closed translation-invariant subalgebra of $\FSa(G)$ is conjugation-closed and has the form
\[
	\set{f\in \FSa(G)\colon\ f\ \textrm{is constant on cosets of}\ H}\quad\textrm{for some normal subgroup $H$ of $G$}\,.
\]
\end{corollary}
\begin{proof}
Let $B$ be a weak$^*$-closed translation-invariant subalgebra of $\FSa(G)$. Set 
\[
	H:=\set{t\in G\colon\ f(t)=f(1)\ (f\in B)}\,.
\]
Then it is easy to see that $H$ is a normal subgroup of $G$, and that $B$ is a weak$^*$-closed translation-invariant subalgebra of $A$, the subalgebra of $\FSa(G)$ consisting of those functions that are constant on cosets of $H$. Thus, by identifying the latter with $\FSa(G/H)$, $B$ becomes a weak$^*$-closed translation-invariant subalgebra of $\FSa(G/H)$ that separates points of $G/H$. Since $G/H$ is (discrete and) amenable, it follows from the previous lemma that $B=\FSa(G/H)$. Considering $B$ as a subalgebra of $\FSa(G)$ again, this means that $B=A$, which is conjugation-closed.
\end{proof}

Put this differently, we have

\begin{corollary}\label{weak*-closed invariant subalgebra of B(G) when G is discrete 3}
Let $G$ be any amenable discrete group. Then every sub-semidual of $G$ is actually a subdual. \enproof
\end{corollary}

Returning to locally compact groups, we obtain the following.

\begin{proposition}\label{when a sub-semidual is thick}
Let $G$ be a locally compact group that contains an open subgroup $H$ with $H_d$ amenable, and let $\Scal$ be a sub-semidual of $G_d$. Then the following are equivalent:
\begin{enumerate}
	\item $\Scal$ is thick.
	\item $\lambda_G\preccurlyeq \Scal$ {\rm(}as representations of $G_d${\rm)}.
	\item $\FSa_r(G)\subseteq \FSa_\Scal^c(G)$.
\end{enumerate}
\end{proposition}
\begin{proof}
(i)$\Rightarrow$(ii): By Lemma \ref{weak*-closed invariant subalgebra of B(G) when G is discrete}, we see that $\lambda_{G_d}\preccurlyeq\Scal$. Condition (ii) then follows from \cite[Theorem 2]{BKLSchl}.

The rest is similar to the proof of Proposition \ref{when a subdual is thick}.
\end{proof}

\section{Characterisations of Fourier algebras}

\noindent In this section, we shall establish our most fundamental key result of this paper.

\begin{theorem}\label{a characterisation of Fourier algebras}
Let $A$ be an $F$-algebra that is also a Tauberian subalgebra of $\C(\Omega)$ that approximately contains a nontrivial real function, for some topological space $\Omega$. Suppose also that:
\begin{enumerate}
	\item every character of $A$ is implemented by some element of $\Omega$;
	\item $\Omega$ is a group and $A$ is left translation-invariant; 
	\item $\norm{\sum_{i=1}^m \alpha_i L_{s_i}:A\to A}\le 1$ whenever $\alpha_i\in\complexs$ and $s_i\in \Omega$ with
		\[	
			\abs{\sum_{i=1}^m\alpha_i f(s_i)}\le \norm{f}\quad(f\in A)\,.
		\]
\end{enumerate}
Then $A\cong \Fa(G)$ for some locally compact group $G$.
\end{theorem}

\begin{remark}
We remark that $\Omega$ is not required a priory to be a topological group and the topology on $\Omega$ is not required to be locally compact.
\end{remark}

\begin{proof}
Condition (i) implies that 
\begin{itemize}
	\item[(i.a)] there is an $e\in \Omega$ such that $\duality{1}{f}=f(e)$ for every $f\in A$, and
	\item[(i.b)] for each $s\in \Omega$ there is a $t\in \Omega$ such that $f^*(s)=\conjugate{f(t)}$ for all $f\in A$;
\end{itemize}
where $1$ is the identity of $\dual{A}$. The first implication is obvious, and for the second one, notice that, since the set $\set{f\in A\colon f(e)=\norm{f}}$ of positive normal functionals of $\dual{A}$ is closed under multiplication, the mapping $f\mapsto f^*, A\to A,$ is multiplicative. Hence, for each $s\in\Omega$, the functional $f\mapsto \conjugate{f^*(s)}$ is multiplicative on $A$, and so is induced by some $t\in\Omega$ by (i). Until the last paragraph, these two conditions (i.a) and (i.b) will be the ones we work with instead of (i).

For each $s\in G$, then
\[
	\norm{f}\ge \abs{f(s)}\quad(f\in A)\,,
\] 
and so, it follows from (iii) that ${L_s}:{A}\to{A}$ is contractive. Define on $\Omega$ a new product $s\diamond t:= se^{-1}t$. Then obviously $(\Omega,\diamond)$ is still a group, and $A$ is still left translation-invariant with respect to this group structure; the left translate by an element $s$ with respect to the new product $\diamond$ is $L^\diamond_s:=L_{e^{-1}}L_s$. The proven  contractivity of $L_{e^{-1}}$ then implies that condition (iii) remains hold with $L^\diamond$ replacing $L$. Thus replacing the old product on $\Omega$ by the new product $\diamond$ if necessary, we shall suppose from now on that $e=1$ is the identity of $\Omega$.

Since the Banach algebra $A$ is a subalgebra of $\C(\Omega)$, there is a natural continuous map $\eta:\Omega\to\spectrum(A)$. Denote by $N_0$ the linear span of $\eta(\Omega)$ in $\dual{A}$. We define a (new) product on $N_0$ as follows:
\[
	\left(\sum_{i=1}^m\alpha_i \eta(s_i)\right)\bullet\left(\sum_{j=1}^n\beta_j \eta(t_j)\right):= \sum_{i=1}^m\sum_{j=1}^n\alpha_i\beta_j\eta(s_it_j)\,.
\]
The well-definedness of this formula follows from the following (in)equalities
\begin{align*}	
	&\sum_{i=1}^m\sum_{j=1}^n\alpha_i\beta_j\eta(s_it_j)=\left(\sum_{i=1}^m \alpha_i \dual{L_{s_i}}\right)\left(\sum_{j=1}^n\beta_j \eta(t_j)\right)\\
	\textrm{and}\quad &\norm{\left(\sum_{i=1}^m \alpha_i \dual{L_{s_i}}\right)\left(\sum_{j=1}^n\beta_j \eta(t_j)\right)}\le \norm{\sum_{i=1}^m\alpha_i \eta(s_i)}\norm{\sum_{j=1}^n\beta_j \eta(t_j)}
\end{align*}
for every $\alpha_i,\beta_j\in\complexs$ and $s_i,t_j\in \Omega$. Indeed, by scaling, we may (and shall) assume that $\norm{\sum_{i=1}^m\alpha_i \eta(s_i)}\le 1$. Then, for every $f\in A$, we see that
\[
	\norm{f}\ge \abs{\duality{\sum_{i=1}^m\alpha_i \eta(s_i)}{f}}=\abs{\sum_{i=1}^m\alpha_i f(s_i)}\,,
\]
and so, by (iii), 
\[
	\norm{\sum_{i=1}^m \alpha_i \dual{L_{s_i}}:\dual{A}\to \dual{A}}=\norm{\sum_{i=1}^m \alpha_i L_{s_i}:A\to A}\le 1\,.
\]
Thus it follows that the product $\bullet:N_0\times N_0\to N_0$ is well-defined, bilinear with norm $1$, and associative. Extending this product by completion to $N:=\closure{N_0}$, we obtain a Banach algebraic structure on $N$, whose multiplication is not necessarily the one induced from that of $\dual{A}$ as the dual of the $F$-algebra $A$. However, note that $N$ is a unital algebra whose unit is $\eta(1)=1$, the unit of $\dual{A}$.

Condition (i.b) implies that $\eta(s)^*\in \eta(\Omega)$ for every $s\in \Omega$, and so $N^*=N$. This \emph{does not} mean (yet) that $N$ is a $*$-algebra, but we still see that  
\[
	N=N_h+\init N_h
\]
where $N_h:=\set{x\in N\colon x=x^*}$. We \emph{claim} that every $x\in N_h$ is a hermitian element of the Banach algebra $N$ as defined using numerical range in \cite[Definition 10.12]{BD}. Indeed, let $x\in N_h$ and take $\omega\in \dual{N}$ with 
\[
	\omega(1)=\norm{\omega}=1\,.
\]
By the Hahn--Banach theorem, we extend $\omega$ to a functional $\tilde{\omega}\in\dual{(\dual{A})}$ such that $\norm{\tilde{\omega}}=\norm{\omega}=1$. Since we still have $\tilde{\omega}(1)=1$, it follows that $\tilde{\omega}$ is a state of the \cstar-algebra $\dual{A}$. But then
\[
	\omega(x)=\tilde{\omega}(x)\in\reals\,.
\]
This shows that $x$ is a hermitian element of the Banach algebra $N$. The Vidav--Palmer theorem \cite{Palmer} (see \cite[Theorem 38.14]{BD} for a proof) then implies that $N$ is a \cstar-algebra whose involution (but \emph{not} the product) is the one induced from $\dual{A}$. In fact, the argument above also shows that if $x$ is positive in $N$, then it is positive in $\dual{A}$. That is, the injection map $\iota:N\to\dual{A}$ is a positive linear map.

As shown in the first paragraph, $\dual{L_s}:\dual{A}\to\dual{A}$ is contractive for each $s\in\Omega$. But since
\[
	\dual{L_s}\dual{L_{s^{-1}}}=\dual{L_{s^{-1}}}\dual{L_s}=\dual{L_1}=\id_{\dual{A}}\,,
\]
it follows that $\dual{L_s}$ is an isometry from $\dual{A}$ onto itself. It follows from the structure theory of isometries of unital \cstar- algebras \cite{Kadison51} that $\dual{L_s}(1)$ is a unitary of $\dual{A}$. But $\dual{L_s}(1)=\dual{L_s}\eta(1)=\eta(s)$, it follows that
\[
	\eta(\Omega)\subseteq \Ucal(\dual{A}).
\]
On the other hand, the map $s\mapsto \eta(s)$ is a group homomorphism from $\Omega$ into the invertible group of $N$. But every element of $\eta(\Omega)$, being an element of $\spectrum(A)$, must have norm $1$, and the only norm one invertible elements of a unital \cstar-algebra are its unitaries, we obtain that
\[
	\eta(\Omega)\subseteq \Ucal(N)\cap \Ucal(\dual{A})\,,
\]
and that $\eta(s)^*=\eta(s^{-1})$ for every $s\in \Omega$. Let $s,t\in \Omega$ and $\alpha,\beta\in\complexs$ be arbitrary. Set
\[
x:=\alpha\eta(s)+\beta\eta(t)+\conjugate{\alpha}\eta(s)^*+\conjugate{\beta}\eta(t)^*.
\]
Then $x=x^*\in N$, and so Kadison's generalized Schwarz inequality \cite[Theorem 2]{Kadison52}, applied to the inclusion $\iota:N\to \dual{A}$, gives  $x\bullet x\ge xx$. We see that
\begin{align*}
    x\bullet x=2\Re\bigg\{ & \abs{\alpha}^2+\abs{\beta}^2+
    \alpha^2\ \eta(s^2)+\beta^2\ \eta(t^2)+\\
    &
    \alpha\beta\ [\eta(st)+\eta(ts)]+\alpha\conjugate{\beta}\ [\eta(st^{-1})+\eta(t^{-1}s)]\bigg\}\quad \textrm{and}\\
    x x= 2\Re\bigg\{ & \abs{\alpha}^2+\abs{\beta}^2 + \alpha^2\eta(s)\eta(s)+\beta^2\eta(t)\eta(t)+\\
    &
    \alpha\beta[\eta(s)\eta(t)+\eta(t)\eta(s)]+\alpha\conjugate{\beta}[\eta(s)\eta(t)^*+\eta(t)^*\eta(s)]\bigg\}.
\end{align*}
It follows (cf. the proof of \cite[Lemma 4.2]{Pham}) that
\begin{align*}
    \eta(st)+\eta(ts)=\eta(s)\eta(t)+\eta(t)\eta(s)\quad(s,t\in \Omega)\,.
\end{align*}
Thus, by linearity and density, we see that 
\[
	x\bullet y+y\bullet x=xy+yx\quad(x,y\in N)\,,
\]
i.e. $\iota:N\to \dual{A}$ is a Jordan-homomorphism. By \cite[Lemma 3.1 and Theorem 3.3]{Stormer}, $\iota$ extends to a weak$^*$-continuous Jordan-homomorphism $\kappa:\bidual{N}\to \dual{A}$, whose kernel must be a (weak$^*$-closed) two-sided ideal of $\bidual{N}$. Thus $\kappa$ induces a weak$^*$-continuous Jordan-monomorphism $\kappa_0$ from a von Neumann algebra $\bidual{N}/\ker\kappa$ into $\dual{A}$. The map $\kappa_0$ must be isometric by \cite[Corollary 3.5]{Stormer}, and hence its image is weak$^*$-closed. Set $G:=\eta(\Omega)$. Since $f=(\hat{f}|_G)\circ\eta$ for every $f\in A$, the map $f\mapsto \hat{f}|_G$ ($f\in A$) is injective, and so $G$ spans a weak$^*$-dense subspace of $\dual{A}$. But $G$ is in the image of $\kappa_0$, so $\kappa_0$ is a Jordan isomorphism from $\bidual{N}/\ker\kappa$ onto $\dual{A}$. In particular, the product $\bullet$ of $N$ extends to a \emph{new} product on $\dual{A}$, also denoted by $\bullet$, making $\dual{A}$ a \emph{new} von Neumann algebra (whose unit is still the given unit and the adjoint operation is still the given one).

Now $G$ with the new product $\bullet$ of $\dual{A}$ and with the weak$^*$-topology of $\dual{A}$ is a topological group; this is because the unitary group $\Ucal(\dual{A})$ is a topological group with respect to the weak$^*$-topology. Since $A$ is isomorphic to $\widehat{A}|_G$ via the map $f\mapsto \hat{f}|_G$, we identify $A$ with $\widehat{A}|_G$, and it is easy to see that $A$ is also Tauberian as a subalgebra of $\C(G)$. Take an $f\in A$ such that $f(s)=1$ for some $s\in G$. Since $A$ is a Tauberian  subalgebra of $\C(G)$, we can actually choose such $f$ with a compact support in $G$. It follows that $\set{t\in G\colon \abs{f(t)}\ge 1/2}$ is a compact neighbourhood of $s$ in $G$. Thus $G$ is a locally compact group.

Let $f\in P(A)$. We \emph{claim} that $f \in \FSa(G)$. Indeed, take $\alpha_1,\ldots,\alpha_n\in\complexs$ and $t_1,\ldots, t_n\in G$. Then, since $\iota:N\to \dual{A}$ is positive, we see that
\begin{align*}
0 &\le \duality{\left(\sum_{j=1}^n \alpha_j \eta(t_j)\right)^*\bullet\left(\sum_{j=1}^n \alpha_j \eta(t_j)\right)}{f}\\
&=\sum_{i,j=1}^n \conjugate{\alpha_i}\alpha_j\duality{\eta(t_i^{-1}t_j)}{f}=\sum_{i,j=1}^n \conjugate{\alpha_i}\alpha_j f(t_i^{-1}t_j)\,.
\end{align*}
This shows that $f$ is positive definite. Thus $A\subseteq \FSa(G)$. Since $\FSa(G)$ is (commutative and) semisimple, it follows that the inclusion $A\to \FSa(G)$ is bounded by \v{S}hilov's  theorem (cf. \cite[Theorem 2.3.3]{Dales}). The condition that $A\cap \C_c(G)$ is dense in $A$ now implies that in fact $A\subseteq \Fa(G)$. Moreover, because now translations by elements of $G$ are precisely module actions of $G$ as a subset of $(\dual{A},\bullet)$ on $A$, $A$ is (two-sided) translation-invariant as an algebra of functions on the group $G$. Consider the inclusion $\Phi:A\to  \Fa(G)$, its dual $\dual{\Phi}:\VN(G)\to \dual{A}$ maps $\lambda_G(t)\mapsto t$ for each $t\in G$. It follows from the discussion above that $\dual{\Phi}$ is a (normal) $*$-homomorphism when $\dual{A}$ is with the product $\bullet$, which must have a weak$^*$-dense range in $\dual{A}$ because $\Phi$ is injective. Thus, in fact, $\dual{\Phi}:\VN(G)\to\dual{A}$ is a $*$-epimorphism, and so $\Phi$ is an isometric monomorphism from $A$ into $\Phi(A)$.

Hence $A$ is a closed translation-invariant Tauberian subalgebra of $\Fa(G)$. Under condition (i), we also have $\spectrum(A)=G$. By Theorem \ref{when A=A(G), nonzero real function case} (and also Lemma \ref{real element vs real function}), we obtain $A=\Fa(G)$. 
\end{proof}

\begin{remark}\label{a characterisation of Fourier algebras: remark}
If in the hypothesis of the previous theorem, instead of condition (i), we assume only (i.a) and (i.b), then it is no longer true that $A=\Fa(G)$; the closure of polynomials in $\Fa(\unitcircle)$ provides an example. In this case, in order to still conclude that $A=\Fa(G)$, we could replace the assumption that $A$ (approximately) contains a nonzero real function by the stronger condition that $A$ is conjugation-closed.
\end{remark}

\begin{corollary}\label{a characterisation of Fourier algebras 2}
Let $A$ be a  commutative Banach algebra whose dual $\dual{A}$ is a \wstar-algebra and whose spectrum $\spectrum(A)$ contains a subset $G$ that is a group under the multiplication of $\dual{A}$ and is separating for $A$. Suppose also that $\widehat{A}|_G$ is a conjugation-closed Tauberian algebra on $G$. Then $\spectrum(A)=G$ is a locally compact group, and the Gelfand transfrom is an isometric isomorphism from $A$ onto $\Fa(G)$.
\end{corollary}
\begin{proof}
Since $G$ is separating for $A$, it spans a weak$^*$-dense subspace of $\dual{A}$, and so, it is easy to see that the identity of $G$ must be $1$, the identity of $\dual{A}$. In particular, $A$ is an $F$-algebra. Also $G$ is a subgroup of the invertible group of $\dual{A}$, but since $G\subseteq\spectrum(A)\subseteq \dual{A}_{[1]}$, $G$ must actually be a subgroup of the unitary group of $\dual{A}$. In particular, $G$ is self-adjoint.

The assumption also implies that $f\mapsto \hat{f}|_G$ is an algebra isomorphism from $A$ onto $\widehat{A}|_G$. Giving the latter the norm induced from $A$, we may (and shall) assume that $A\subseteq \C(G)$. It is also easy to see that $L_sf=f\cdot s$, the right $\dual{A}$-module action of $s$ on $f$, for each $f\in A$ and $s\in G$. So $A$ is left translation-invariant, and for every $\alpha_i\in\complexs$ and $s_i\in G$ with $\abs{\sum_{i=1}^m\alpha_i f(s_i)}\le \norm{f}$ $(f\in A)$, we see that 
\[
	\norm{\sum_{i=1}^m \alpha_i L_{s_i}:A\to A}=\norm{\sum_{i=1}^m \alpha_i s_i}\le 1\,.
\]
Thus the hypothesis of Theorem \ref{a characterisation of Fourier algebras} holds (with condition (i) being replaced by conditions (i.a) and (i.b) in its proof). We obtain that $A$ is a conjugation-closed translation-invariant closed subalgebra of $\Fa(G)$ that separates points of $G$. It follows from \cite[Theorem 2.1]{BLSchl}  that $A=\Fa(G)$.
\end{proof}

The following generalises \cite[Theorem 3.2.12]{CL}.

\begin{corollary}\label{a characterisation of Fourier algebras 3}
Let $A$ be a  commutative semisimple Banach algebra that is also Tauberian and approximately contains a nontrivial real element. Suppose that $\dual{A}$ is a \wstar-algebra and that $\spectrum(A)$, under the multiplication of $\dual{A}$, is a group. Then $A\cong \Fa(G)$ for some locally compact group $G$. \enproof
\end{corollary}

We have the following characterisation of when the predual of a Hopf--von Neumann algebra is the Fourier algebra of a locally compact group.

\begin{corollary}\label{a characterisation of Fourier algebras, Hopf--von Neumann case}
Let $A$ be the predual of a Hopf--von Neumann algebra such that $A$ is  commutative, semisimple, Tauberian, and approximately contains a nontrivial real element, and that $\spectrum(A)$ has at most one positive element. Then $A\cong \Fa(G)$ for some locally compact group $G$.
\end{corollary}
\begin{proof}
Denote by $\Delta:\dual{A}\to\dual{A}\closure{\otimes}\dual{A}$ the comultiplication. Then
\[
	\spectrum(A)=\set{x\in\dual{A}\setminus\set{0}\colon \Delta(x)=x\otimes x}\,.
\]
In particular, $1\in\spectrum(A)$, and, by the assumption, it is the only positive element in $\spectrum(A)$. Moreover, the above equation also shows that $\spectrum(A)\cup\set{0}$ is self-adjoint and multiplicative. Take $x\in\spectrum(A)$. It follows that $xx^*$ and $x^*x$ also belong to $\spectrum(A)$, and so $x^*x=xx^*=1$. Thus $\spectrum(A)$ is a subgroup of the unitary group of $\dual{A}$. The result then follows from the previous corollary.
\end{proof}

The following example shows that we cannot omit from the above corollary the assumption that $\spectrum(A)$ has at most one positive element.

\begin{example}
Take $n\in\naturals$ with $n\ge 2$. Let $A:=\lone_n$ so that $\dual{A}=\linfty_n$ is an $n$-dimensional commutative von Neumann algebra. The standard basis for both $\lone_n$ and $\linfty_n$ is denoted by $\set{\delta_1,\ldots,\delta_n}$. Define a linear map $\Delta:\linfty_n\to \linfty_n\otimes\linfty_n$ by setting
\[
	\Delta(\tuple{1}):=\tuple{1}\otimes \tuple{1}\quad \textrm{and}\quad \Delta(\delta_j):=\delta_j\otimes\delta_j\quad(2\le j\le n)\,;
\]
where $\tuple{1}:=\sum_{j=1}^n\delta_j$. It is easy to see that $(\linfty_n,\Delta)$ then becomes a Hopf--von Neumann algebra. Endowed with the induced multiplication, $A$ becomes an $n$-dimensional Banach algebra whose spectrum contains $\set{\tuple{1}, \delta_2,\ldots,\delta_n}$. It follows that $A$ is commutative and semisimple and that $\spectrum(A)=\set{\tuple{1}, \delta_2,\ldots,\delta_n}$. The conjugation-closed and Tauberian properties of $A$ now follow easily.
\end{example}

The identity element of a unital semisimple commutative Banach algebra $A$ is always a (nonzero) real element, and moreover, in this case, $A$ is automatically Tauberian. Thus we have the following characterisations of $\Fa(G)$ for compact $G$.

\begin{corollary}\label{a characterisation of Fourier algebras, compact case}
Let $A$ be a unital $F$-algebra that is also a subalgebra of $\C(\Omega)$, for some topological space $\Omega$. Suppose also that:
\begin{enumerate}
	\item every character of $A$ is implemented by some element of $\Omega$;
	\item $\Omega$ is a group and $A$ is left translation-invariant; 
	\item $\norm{\sum_{i=1}^m \alpha_i L_{s_i}:A\to A}\le 1$ whenever $\alpha_i\in\complexs$ and $s_i\in \Omega$ with
		\[	
			\abs{\sum_{i=1}^m\alpha_i f(s_i)}\le \norm{f}\quad(f\in A)\,.
		\]
\end{enumerate}
Then $A\cong \Fa(G)$ for some compact group $G$. \enproof
\end{corollary}

\begin{corollary}\label{a characterisation of Fourier algebras, Hopf--von Neumann case, compact}
Let $A$ be the predual of a Hopf--von Neumann algebra such that $A$ is  unital, commutative, and semisimple, and that $\spectrum(A)$ has at most one positive element. Then $A\cong \Fa(G)$ for some compact group $G$. \enproof
\end{corollary}


\section{Special isometries and characterisations of Fourier algebras}
\label{Special isometries and characterisations of Fourier algebras}

\noindent Walter proved in \cite{Walter74} that isometric automorphisms $T$ of $\FSa(G)$ (or $\Fa(G)$) with the property that
\[
	\norm{f-\emath^{\init\theta}Tf}^2+\norm{f+\emath^{\init\theta}Tf}^2\le 4\norm{f}^2\qquad(f\in \FSa(G),\ \theta\in\reals)
\]
are precisely the left translations and right translations by elements of $G$. He then went on to use this inequality among other axioms to characterise the Fourier (and the Fourier--Stieltjes) algebras in \cite[\S 4]{Walter74}. Our aim now is to use this inequality as a replacement for condition (iii) in Theorem \ref{a characterisation of Fourier algebras}, to obtain a different simple set of axioms that characterise the Fourier algebras; the Fourier--Stieltjes algebras case will be dealt with in a later section.

In fact, it follows from the parallelogram law for Hilbert spaces that a linear isometry $T$ of the predual $\predual{M}$ of a von Neumann algebra $M$ will satisfy
\begin{align}\label{eq: a nice isometry}
	\norm{f-\emath^{\init\theta}Tf}^2+\norm{f+\emath^{\init\theta}Tf}^2\le 4\norm{f}^2 
\end{align}
for every $f\in \predual{M}$ and $\theta\in\reals$, if it is a left or right module multiplication by a unitary of $M$. More generally, if $T$ is a \emph{mixed module multiplication} by a unitary of $M$, i.e., if there is a central projection $z\in M$ and a unitary $u\in M$ such that
\[
	Tf=uz\cdot f+f\cdot(1-z)u\qquad (f\in \predual{M})\,,
\]
then $T$ is a linear isometry of $\predual{M}$ that satisfies the above inequality. However, these are not all of such isometries. Below, whenever $E$ is a subspace of a Banach space $X$, we shall write
\[
	E^\perp:=\set{x'\in\dual{X}\colon\quad \duality{x'}{x}=0\ (x\in E)}\,.
\]
Also, when we say that $T$ is an isometry \emph{of} $X$, we mean that $T$ is also surjective.

\begin{lemma}\label{a general description of nice isometry}
Let $I$ and $J$ be two closed ideals in a von Neumman algebra $M$, and let $u,v\in M$ be unitaries. Suppose that $I\cap J=\set{0}$ and that $vx-xu\in I+J$ for every $x\in M$. Then there is a unique bounded linear operator $T:\predual{M}\to\predual{M}$ such that
\[
	\bidual{T}f=u\cdot f\quad\textrm{for every}\ f\in I^\perp\quad\textrm{and}\quad \bidual{T}f=f\cdot v\quad\textrm{for every}\ f\in J^\perp\,.
\]
Moreover, $T$ is an isometry of $\predual{M}$ that satisfies \eqref{eq: a nice isometry} for  every $f\in \predual{M}$ and $\theta\in\reals$.
\end{lemma}
\begin{proof}
Since $I\cap J=\set{0}$, we see that $\dual{M}=I^\perp+J^\perp$; this is due to the Hahn--Banach theorem and the fact that the natural mapping $M\to M/I\oplus M/J$ is an injective $*$-homomorphism, and hence isometric.  Since $I^\perp\cap J^\perp=(I+J)^\perp$, we see from the assumption that $u\cdot f=f\cdot v$ whenever $f\in I^\perp\cap J^\perp$. Thus there exists a unique linear mapping $S:\dual{M}\to\dual{M}$ such that
\[
	Sf=u\cdot f\quad\textrm{for every}\ f\in I^\perp\quad\textrm{and}\quad Sf=f\cdot v\quad\textrm{for every}\ f\in J^\perp\,.
\]
Finally, since $I^\perp$ is closed sub-$M$-bimodule of $\dual{M}$ (it is in fact even weak$^*$-closed),  it has a support projection $z_I$ in $\bidual{M}$, which is a central projection in $\bidual{M}$ such that $I^\perp=z_I\cdot \dual{M}$  (cf. \cite[Theorem III.2.7]{Takesaki}); $1-z_I$ is actually the unit of the $\sigma(\bidual{M},\dual{M})$-closure of $I$. Let $z_J$ be the similar projection for $J$. Then $z_I+z_J\ge 1$, and so $(1-z_I)\cdot\dual{M}\subseteq J^\perp$. It then follows that
\[
	Sf=uz_I\cdot f+f\cdot(1-z_I)v\quad (f\in\dual{M})\,.
\]
In particular, $S$ is an isometry of $\dual{M}$ that satisfies \eqref{eq: a nice isometry} for  every $f\in \dual{M}$ and $\theta\in\reals$. It remains to prove the existence of an isometry $T$ on $\predual{M}$ such that $\bidual{T}=S$. For this it is sufficient to prove that $S:\dual{M}\to\dual{M}$ is $\sigma(\dual{M},M)$-continuous, because if that is the case then $S=\dual{\predual{S}}$ for some bounded linear operator $\predual{S}:M\to M$, which is necessarily isometric and onto, and so, by \cite{Kadison51}, $\predual{S}$ is in turn $\dual{T}$ for some isometry $T$ on $\predual{M}$. To prove that $S:\dual{M}\to\dual{M}$ is $\sigma(\dual{M},M)$-continuous, it is sufficient by the Kre\u{\i}n--\u{S}mulian theorem to prove that $S:\closedball{\dual{M}}{1}\to\dual{M}$ is $\sigma(\dual{M},M)$-$\sigma(\dual{M},M)$-continuous. Assume toward a contradiction that there is a net $(f_\alpha)$ in $\closedball{\dual{M}}{1}$ that $\sigma(\dual{M},M)$-converges to $f$ but such that $(Sf_\alpha)$ does not $\sigma(\dual{M},M)$-converge to $Sf$. Write $f_\alpha=g_\alpha+h_\alpha$ where $g_\alpha\in \closedball{I^\perp}{1}$ and $h_\alpha\in \closedball{J^\perp}{1}$, says $g_\alpha:=z_I\cdot f_\alpha$ and $h_\alpha:=(1-z_I)\cdot f_\alpha$. Since $I^\perp$ and $J^\perp$ are $\sigma(\dual{M},M)$-closed, by passing to subnets, we shall assume that $(g_\alpha)$ $\sigma(\dual{M},M)$-converges to some $g$ and $(h_\alpha)$ $\sigma(\dual{M},M)$-converges to some $h$. But then because $u,v\in M$, we see from the above decomposition formula for $S$ that $(Sf_\alpha)$ $\sigma(\dual{M},M)$-converges to $S(g+h)=Sf$; a contradiction.
\end{proof}

\begin{remark}
With the assumption as in the previous lemma, but suppose moreover that $I+J$ is the closed ideal generated by $\set{vx-xu\colon x\in M}$, noting that $I+J$ is always a closed ideal of $M$ (cf. \cite[Exercise I.8.4]{Takesaki}), then it can be seen that
\begin{itemize}
	\item $I^\perp$ is the largest closed sub-$M$-bimodule of $\dual{M}$ on which $\bidual{T}f=u\cdot f$, and
	\item $J^\perp$ is the largest closed sub-$M$-bimodule of $\dual{M}$ on which $\bidual{T}f=f\cdot v$.
\end{itemize}
In particular, if in addition the $\sigma(M,\predual{M})$-closures of $I$ and $J$ intersects, then the constructed isometry $T$ will not be a mixed module multiplication by any unitary of $M$ (although, $\bidual{T}$ always is a mixed module multiplication as a map on the predual of $\bidual{M}$).
\end{remark}

To prove the converse of Lemma \ref{a general description of nice isometry}, we need the following strengthening of \cite[Theorem 2]{Walter74}.

\begin{lemma}\label{a preliminary result by Walter}
Let $T$ be a linear isometry of the predual $\predual{M}$ of a von Neumann algebra $M$ that satisfies \eqref{eq: a nice isometry} for every normal \emph{pure} state $f$ and every real number $\theta$. Suppose that $z$ is a minimal central projection of $M$. Then
\begin{enumerate}
	\item $\dual{T}(zx)=z\dual{T}(x)$ for every $x\in M$. In particular, $\dual{T}(zM)=zM$.
	\item If $z$ is moreover of type I, then either $\dual{T}(x)=ux$ for every $x\in zM$ or $\dual{T}(x)=xu$ for every $x\in zM$, where $u:=\dual{T}(1)$.
\end{enumerate}
\end{lemma}
\begin{proof}
(i) By \cite{Kadison51}, we have $\dual{T}=u\Psi$, where $\Psi$ is a Jordan $*$-isomorphism of $M$. Arguing as at the bottom of \cite[page 136]{Walter74}, we see that $\Psi(z)=z$. Hence, since $z$ is central, $\dual{T}(zx)=u\Psi(zx)=uz\Phi(x)=z\dual{T}(x)$.

(ii) By (i), we may (and shall) suppose that $z=1$. The assumption then implies that $M$ is a factor of type I; says $M=\operators(\Hcal)$ for some Hilbert space $\Hcal$. Also, it follows from \cite{Kadison51} that in this case $\Psi$ is either a $*$-isomorphism or a $*$-anti-isomorphism of $M$. The argument on \cite[pages 137--138]{Walter74}, with obvious modifications, then implies that for every $\xi\in \Hcal$
\begin{itemize}
	\item either $\inner{(\dual{T}x)\xi}{\xi}=\inner{ux\xi}{\xi}$ for every $x\in M$ (called Case I),
	\item or $\inner{(\dual{T}x)\xi}{\xi}=\inner{xu\xi}{\xi}$ for every $x\in M$ (called Case II).
\end{itemize}

Assume towards a contradiction that there exist a vector $\xi_1\in \Hcal$ for which Case I does not hold and a vector $\xi_2\in \Hcal$ for which Case II does not hold. Then since
\begin{align*}
	\inner{(\dual{T}x)(\xi_1+t\xi_2)}{(\xi_1+t\xi_2)}+\inner{(\dual{T}x)(\xi_1-t\xi_2)}{(\xi_1-t\xi_2)}\\
	=2\inner{(\dual{T}x)\xi_1}{\xi_1} + 2t^2\inner{(\dual{T}x)\xi_2}{\xi_2}
\end{align*}
for every $x\in M$ and $t\in\reals$, we see that, for each $t\neq 0$, neither Case I nor Case II can hold simultaneously for the two vectors $\xi_1\pm t\xi_2$. It follows that, for infinitely many $t\neq 0$, Case I holds for $\xi_1+t\xi_2$ and Case II holds for $\xi_1-t\xi_2$. Thus we obtain
\begin{align*}
	\inner{ux(\xi_1+t\xi_2)}{(\xi_1+t\xi_2)}+\inner{xu(\xi_1-t\xi_2)}{(\xi_1-t\xi_2)}\\
	=2\inner{xu\xi_1}{\xi_1} + 2t^2\inner{ux\xi_2}{\xi_2}
\end{align*}
for every $x\in M$ and, for infinitely many $t\neq 0$, and hence, for all  $t\in\reals$. But then letting $t=0$, we see that both Case I and Case II hold for $\xi_1$; a contradiction.

Thus we must have either Case I hold for all $\xi\in \Hcal$ or Case II hold for all $\xi\in \Hcal$. From there, the result follows easily.
\end{proof}

Note that in the previous lemma, the inequality \eqref{eq: a nice isometry} is only required to hold for normal pure states. When $\predual{M}$ is $\FSa(G)$ or $\Fa(G)$, $M$ has a plenty of normal pure states as demonstrated in Walter's refinement \cite[Proposition 2]{Walter74} of the Gelfand--Raikov theorem and this fact is utilised in his characterisations of the left and right translations on $\FSa(G)$ and $\Fa(G)$ in \cite{Walter74}. In general, $M$ may not have any normal pure state, and that is why we need to extend the effective range of the inequality. 

We now have  the converse of Lemma \ref{a general description of nice isometry}.

\begin{lemma}\label{characterisation of a single nice isometry}
Let $T$ be an isometry of the predual $\predual{M}$ of a von Neumann algebra $M$ that satisfies \eqref{eq: a nice isometry} for every $f\in\predual{M}^+$ and $\theta\in\reals$. Set $u:=\dual{T}1$, an unitary of $M$. Then there exist two closed ideals  $I$ and $J$ in $M$ such that $I\cap J=\set{0}$, that $I+J$ is the closed ideal generated by $\set{ux-xu\colon x\in M}$, and such that
\[
	\bidual{T}f=u\cdot f\quad\textrm{for every}\ f\in I^\perp\quad\textrm{and}\quad \bidual{T}f=f\cdot u\quad\textrm{for every}\ f\in J^\perp\,.
\]
\end{lemma}
\begin{proof}
Since the normal state space of $M$ is $\sigma(\dual{M},M)$-dense in its state space (see, for example, \cite[Theorem 4.3]{Namioka}), we see that 
\begin{align*}
	\norm{f-\emath^{\init\theta}\bidual{T}f}^2+\norm{f+\emath^{\init\theta}\bidual{T}f}^2\le 4\norm{f}^2\qquad(f\in \dual{M}_+,\ \theta\in\reals)\,.
\end{align*}
Of course, $\bidual{T}$ is an isometry of $\dual{M}$. Now, identifying $M$ naturally with a unital $\sigma(\bidual{M},\dual{M})$-dense sub-\cstar-algebra of $\bidual{M}$, so that $\tripledual{T}$ extends $\dual{T}$. In particular, we have
\[
	u:=\dual{T}(1_M)=\tripledual{T}(1_{\bidual{M}})\in M\,.
\]
Denote by $E$ the set of those minimal type I central projections $z$ of $\bidual{M}$ such that $\tripledual{T}(x)=xu$ for every $x\in z\bidual{M}$, and by $F$ a similar set of $z$ but such that $\tripledual{T}(x)=ux$ for every $x\in z\bidual{M}$. Note that, by Lemma \ref{a preliminary result by Walter}, the set of minimal type I central projections of $\bidual{M}$ is precisely $E\cup F$. Let $V$ denote the $\sigma(\dual{M},M)$-closure of $\sum_{z\in E} z\dual{M}$, and $W$ that of $\sum_{z\in F} z\dual{M}$. Then, by the $\sigma(\dual{M},M)$-continuity of $\bidual{T}$, we see that
\[
	\bidual{T}f=u\cdot f\quad\textrm{for every}\ f\in V\quad\textrm{and}\quad \bidual{T}f=f\cdot u\quad\textrm{for every}\ f\in W\,.
\]
Moreover, since $V$ and $W$ are $\sigma(\dual{M},M)$-closed sub-$M$-bimodule of $\dual{M}$, we have $V=I^\perp$ and $W=J^\perp$ for some closed ideals $I$ and $J$ in $M$. 

For each pure state $g$ of $M$, since $g$ is a normal pure state of $\bidual{M}$, there is a minimal type I central projection $z$ of $\bidual{M}$ such that $g=z\cdot g$ (the support projection of the GNS representation of $\bidual{M}$ associated with $g$). It follows that $g\in V\cup W$.  Now let $f$ be a state of $M$. Then, by the Kre\u{\i}n--Milman theorem, $f$ is a $\sigma(\dual{M},M)$-limit of a net of convex combinations of pure states of $M$. It follows that
\[
	f=\sigma(\dual{M},M)-\lim_\alpha(r_\alpha g_\alpha+s_\alpha h_\alpha)
\]
for some nets $(r_\alpha), (s_\alpha)$ in $[0,1]$ with $r_\alpha+s_\alpha=1$ and some nets $(g_\alpha)$ and $(h_\alpha)$ in $\closedball{V}{1}$ and $\closedball{W}{1}$, respectively. By the $\sigma(\dual{M},M)$-compactness of the latter two sets, it follows that $f\in V+W$. Thus $\dual{M}=V+W$ and so $I\cap J=\set{0}$.

Finally, since $u\cdot f=f\cdot u$ for every $f\in V\cap W=(I+J)^\perp$, it follows that $ux-xu\in I+J$ for every $x\in M$; noting that $I+J$ is automatically  closed in $M$ (cf. \cite[Exercise I.8.4]{Takesaki}). On the other hand, take any $f\in \dual{M}$ such that $f(y)=0$ for every $y$ in the closed ideal generated by $\set{ux-xu\colon x\in M}$. Then $f$ generates a $\sigma(\dual{M},M)$-closed sub-$M$-bimodule $L$ of $\dual{M}$ whose elements annihilate $\set{ux-xu\colon x\in M}$, i.e. $u\cdot g=g\cdot u$ for every $g\in L$. Denote by $z_0$ the central projection of $\bidual{M}$ such that $L=z_0\cdot\dual{M}$. Take an extreme point $g$ of $\closedball{L}{1}\cap\dual{M}_+$. It is easy to see that $g$ is a pure state of $M$, and so, as above, $g=z\cdot g$ for some minimal type I central projection $z$ of $\bidual{M}$. It follows easily that $z\le z_0$, and so $z\cdot\dual{M}\subseteq L$. From this, we see that $z$ must belong to $E\cap F$, being already in $E\cup F$. Hence, $g\in V\cap W$. Since $L$ is $\sigma(\dual{M},M)$-closed, the Kre\u{\i}n--Milman theorem then implies that $f\in L\subseteq V\cap W=(I+J)^\perp$. Hence $I+J$ is the closed ideal generated by $\set{ux-xu\colon x\in M}$. The proof is complete.
\end{proof}

\begin{corollary}\label{power of a nice isometry}
Let $T$ be an isometry of the predual $\predual{M}$ of a von Neumann algebra $M$ that satisfies \eqref{eq: a nice isometry} for every $f\in\predual{M}^+$ and $\theta\in\reals$. Then so is $T^k$ and, moreover, $\dual{(T^k)}(1)=\dual{T}(1)^k$  for each $k\in \integers$. 
\end{corollary}
\begin{proof}
This follows from the previous lemma and Lemma \ref{a general description of nice isometry}.
\end{proof}

The previous lemma could actually be generalized further to (certain) collections of isometries.

\begin{theorem}\label{characterisation of a set of nice isometries}
Let $\Tfrak$ be a set of isometries on the predual $\predual{M}$ of a von Neumann algebra $M$ that satisfy \eqref{eq: a nice isometry} for every $f\in\predual{M}^+$ and $\theta\in\reals$. Suppose that $ST$ also satisfies the same inequality for every $S\neq T\in\Tfrak$.  Then there exist two closed ideals  $I$ and $J$ in $M$ such that $I\cap J=\set{0}$, that $I+J$ is the closed ideal generated by 
\[
	\set{\dual{T}(1)x-x\dual{T}(1)\colon T\in\Tfrak,\ x\in M}\,,
\]
and such that
\begin{align*}
	&\bidual{T}f=\dual{T}(1)\cdot f\quad\textrm{for every $T\in\Tfrak$ and}\ f\in I^\perp\\ 
\textrm{and}\quad &\bidual{T}f=f\cdot \dual{T}(1)\quad\textrm{for every $T\in\Tfrak$ and}\ f\in J^\perp\,.
\end{align*}
\end{theorem}
\begin{proof}
The argument is the same as that of Lemma \ref{characterisation of a single nice isometry} using the following refinement of Lemma \ref{a preliminary result by Walter}. Let $z$ be any minimal central projection of type I in $\bidual{M}$. Then, by Lemma \ref{a preliminary result by Walter}, we obtain for every $T\in \Tfrak$
\begin{itemize}
	\item either $\dual{T}x=\dual{T}(1)x$ for every $x\in zM$ (called Case I),
	\item or $\dual{T}x=x\dual{T}(1)$ for every $x\in zM$ (called Case II).
\end{itemize}
However, since either of those two cases also hold for $ST$ for each pair of $S\neq T\in\Tfrak$, it follows easily that actually  either Case I holds for all $T\in\Tfrak$ or Case II holds for all $T\in\Tfrak$.
\end{proof}

\begin{remark}
Similarly as for the case of a single isometry, the conditions in the conclusion of the previous theorem imply that
\begin{itemize}
	\item $I^\perp$ is the largest closed sub-$M$-bimodule of $\dual{M}$ on which $\bidual{T}f=\dual{T}(1)\cdot f$ for every $T\in\Tfrak$, and
	\item $J^\perp$ is the largest closed sub-$M$-bimodule of $\dual{M}$ on which $\bidual{T}f=f\cdot \dual{T}(1)$ for every $T\in\Tfrak$.
\end{itemize}
In particular, if in addition the $\sigma(M,\predual{M})$-closures of $I$ and $J$ intersects, then there could be no central projection $z\in M$ such that all $T\in\Tfrak$ are left multiplications by unitaries on $z\cdot\predual{M}$ and that all $T\in\Tfrak$ are right multiplications by unitaries on $(1-z)\cdot\predual{M}$. From this, it can then be seen that, in this case, at least one isometry in $\Tfrak$ is not a mixed module multiplication on $\predual{M}$.
\end{remark}

The following extends the first half of Corollary \ref{power of a nice isometry}; the second half could also be generalized to \emph{commutative} collections of isometries, but we shall not need it. 

\begin{corollary}\label{group generated from a set of nice isometries}
Let $\Tfrak$ be as in Theorem \ref{characterisation of a set of nice isometries}. Then every isometry in the group of isometries on $\predual{M}$ generated by $\Tfrak$ also satisfies \eqref{eq: a nice isometry} for every $f\in\predual{M}^+$ and $\theta\in\reals$. \enproof
\end{corollary}

For our characterisations of Fourier (and Fourier--Stieltjes) algebras, the following consequence is important.

\begin{corollary}\label{a set of nice isometries 2}
Let $\Tfrak$ be as in Theorem \ref{characterisation of a set of nice isometries}. Then
\[
	\norm{\sum_{i=1}^m \alpha_i T_i}= \norm{\sum_{i=1}^m \alpha_i \dual{T_i}(1)}\qquad (T_i\in \Tfrak,\ \alpha_i\in\complexs)\,.
\]
\end{corollary}
\begin{proof}
Let $I$ and $J$ be closed ideals in $M$ as specified in the theorem. It follows obviously that
\begin{align*}
	\norm{\sum_{i=1}^m \alpha_i \bidual{T_i}f}&\le \norm{\sum_{i=1}^m \alpha_i \dual{T_i}(1)}\norm{f}
\end{align*}
for every $f\in I^\perp\cup J^\perp$. The same inequality holds for every $f\in \dual{M}$, since we can then write $f=g+h$ where $g\in I^\perp$ and $h\in J^\perp$ with $\norm{f}=\norm{g}+\norm{h}$. 
\end{proof}

Before stating our promised characterisation of Fourier algebras, let us introduce the following concept, suggested by Walter's notion of a dual group of a Banach algebra \cite[Definition 5]{Walter74}; more on the latter below.

\begin{definition}
\label{dual automorphism}
Let $A$ be a Banach algebra. An automorphism $T$ of $A$ is \emph{a dual automorphism} if it is isometric and satisfies
\[
	\norm{f-\emath^{\init\theta}Tf}^2+\norm{f+\emath^{\init\theta}Tf}^2\le 4\norm{f}^2 \qquad(f\in A,\ \theta\in\reals)\,.
\]
\end{definition}

\begin{theorem}\label{a characterisation of Fourier algebras 4}
Let $A$ be a Banach algebra that is also a conjugation-closed Tauberian algebra on some topological space $\Omega$. Suppose that \begin{enumerate}
	\item $\dual{A}$ is a \wstar-algebra whose identity is implemented by some element of $\Omega$;
	\item $\Omega$ is a group, $A$ is left translation-invariant, and, for each $s\in\Omega$, the automorphism $L_s$ is dual for $A$.
\end{enumerate}
Then $A\cong \Fa(G)$ for some locally compact group $G$. 
\end{theorem}

\begin{proof}
This follows from Theorem \ref{a characterisation of Fourier algebras} (see also Remark \ref{a characterisation of Fourier algebras: remark}) and the previous results. Noting that condition (i) here, which implies that $A$ is an $F$-algebra,  is just condition (i.a) in the proof of  Theorem \ref{a characterisation of Fourier algebras}; condition (i.b) in that proof now holds automatically thanks to Lemma \ref{power of a nice isometry}. 
\end{proof}

It is possible to replace our condition (ii) by \cite[axiom (ii) on page 155]{Walter74} as in the following. We recall below the  concept of \emph{a dual group} of a Banach algebra introduced in \mbox{\cite[Definition 5]{Walter74};} here, we shall relax the condition slightly by not requiring such groups to be maximal. More importantly, thanks to Lemma \ref{group generated from a set of nice isometries}, we shall be able to extend the concept to semigroups.

\begin{definition}
\label{dual semigroup}
Let $A$ be a Banach algebra. \emph{A dual [semi]group of $A$} is a [semi]group of dual automorphisms of $A$. Note that no topology is imposed on any dual semigroup.
\end{definition}

\begin{remark}
Let $A$ be a Banach algebra.
\begin{enumerate} 
	\item The trivial group $\set{\id_A}$ is always a dual group of $A$, and the union of any chain of dual [semi]groups is again a dual [semi]group. It then follows from Zorn's lemma that \emph{maximal} dual [semi]groups of $A$ always exist; although even maximal dual [semi]group may not be unique.
	\item Any dual semigroup of $A$ acts naturally on $\spectrum(A)$, as any semigroup of automorphisms of $A$ would act on $\spectrum(A)$ by transposition.
	\item If $A$ is an $F$-algebra, then dual groups of $A$ can be generated from less restrictive sets of dual automorphisms (such as dual semigroups) as described in Corollary \ref{group generated from a set of nice isometries}. Hence, a maximal dual semigroup of an $F$-algebra is automatically a (maximal) dual group. These properties are unlikely to be holds by general Banach algebras.
\end{enumerate}
\end{remark}

\begin{theorem}\label{a characterisation of Fourier algebras 4b}
Let $A$ be a  commutative semisimple $F$-algebra that is Tauberian, approximately contains a nontrivial real element, and possesses a dual semigroup that acts transitively on $\spectrum(A)$. Then 
$A\cong \Fa(G)$ for a locally compact group $G$. 
\end{theorem}
\begin{proof}
Denote by $\Omega$ the specified dual semigroup of $A$. It is obvious that we can (and shall) assume that $\Omega$ is a maximal dual semigroup of $A$ (actually, Corollary  \ref{a set of nice isometries 2} and the transitivity condition imply that the given $\Omega$ is already maximal, but this is not essential), and hence, a maximal dual group of $A$ by Lemma \ref{group generated from a set of nice isometries}. Identifying $A$ with $\widehat{A}$, we consider $A$ as an algebra of function on $\spectrum(A)$. The transitivity of the action of $\Omega$ on $\spectrum(A)$ means that $\spectrum(A)=\set{\dual{T}(1)\colon T\in \Omega}$. Corollary  \ref{a set of nice isometries 2} then implies that the mapping $T\mapsto\dual{T}(1),\ \Omega\to \spectrum(A),$ is a bijection; denote the inverse of this mapping as $u\mapsto T_u\,,\ \spectrum(A)\to \Omega$. This bijection transfers to $\spectrum(A)$ the \emph{opposite} group structure of $\Omega$; denoted by $\bullet$ the induced group product on $\spectrum(A)$. For each $u,v\in \spectrum(A)$ and each $f\in A$, we see that
\begin{align*}
	(T_uf)(v)=\duality{T_uf}{\dual{T_v}1}=\duality{T_vT_uf}{1}=\duality{T_{u\bullet v}f}{1}=f(u\bullet v)\,.
\end{align*}
This implies that $A$ is left translation-invariant and that $L_u=T_u$ for every $u\in \spectrum(A)$. The result then follows as for  Theorem \ref{a characterisation of Fourier algebras 4}.
\end{proof}

\begin{remark}
\begin{enumerate}
\item Beside the relaxation from a dual group to a dual semigroup, another important difference between our characterisation above and  Walter's characterisation of the Fourier algebras in \cite{Walter74} is that, in the latter,  the fact that $A$ is a closed subalgebra of $\FSa(G)$ is an immediate consequence of \cite[axiom (v) on page 156]{Walter74}, due to Eymard's generalisation of the Bochner--Schoenberg--Eberlein theorem. 
Note that this and some earlier axioms in \cite{Walter74} also imply that $A$ is an $F$-algebra.

\item Another axiom called \emph{local normality} -- \cite[axiom (vi) on page 156]{Walter74}, which requires that for every compact subset $K$ of $\spectrum(A)$ and every neighbourhood $W$ of $K$  there exists an $f\in A$ such that $\hat{f}=1$ on $K$, $\hat{f}=0$ on $\Omega\setminus W$, and $\abs{\hat{f}}\le 1$ on $W\setminus K$, is stronger than our assumption that $A$ approximately contains a nontrivial real element, cf. Example \ref{nontrivial real function, example 1} (iii). 
\end{enumerate}
\end{remark}

We note below the particularly (even more) simple formulation of the above result for the classes of locally compact abelian groups and of compact groups. 

\begin{corollary}\label{a characterisation of Fourier algebras 4b, abelian}
Let $A$ be a commutative semisimple $F$-algebra that possesses an abelian dual semigroup that acts transitively on $\spectrum(A)$. Then $A\cong \Lone(\Gamma)$ for a locally compact abelian group $\Gamma$. 
\end{corollary}
\begin{proof}
This is immediate from the above; the priory assumption that $A$ approximately contains a nontrivial real element is omitted in this case thanks to the result of Rieffel mentioned in \S\ref{Invariant subalgebras of Fourier algebras}.
\end{proof}

\begin{corollary}\label{a characterisation of Fourier algebras 4b, compact}
Let $A$ be a unital commutative semisimple $F$-algebra 
that possesses a dual semigroup that acts transitively on $\spectrum(A)$. Then $A\cong \Fa(G)$ for some compact group $G$. \enproof
\end{corollary}

\section{Noncommutative $L'$-inducing property and characterisations of Fourier algebras}
\label{Noncommutative L'-inducing property and characterisations of Fourier algebras}

\noindent In this section, we shall extend \cite[Theorem A]{Rieffel} to the noncommutative setting. Rieffel's theorem characterises the group algebra $\Lone(\Gamma)$ for a locally compact \emph{abelian} group $\Gamma$ in terms of a concept of $L'$-inducing characters on Banach algebras, which in turn bases on his concept of a $L$-inducing functional $\phi$ on a Banach space $A$. The latter amounts to $A$ having a lattice structure of an abstract complex $L$-space, so that by a complex version \cite[Proposition 2.3]{Rieffel} of Kakutani's representation theorem of abstract real $L$-spaces, there is a measure space $X$ such that $A\cong\Lone(X)$ and that $\phi$ corresponding to the constant function $\norm{\phi}$ in $\Linfty(X)$, which is the commutative von Neumann algebra dual of $\Lone(X)$. Since there is no such characterisation of the preduals of general von Neumann algebras, we settle with the following definition.

\begin{definition}
Let $A$ be a Banach space, and let $\phi$ be a norm one functional on $A$. Then $\phi$ is \emph{noncommutative $L$-inducing} \footnote{The phrasing ``possibly noncommutative $L$-inducing'' is probably more precise, but we find it too lengthy.}  if there is a product and an adjoint operation on $\dual{A}$ making it a von Neumann algebra with unit $\phi$. 
\end{definition}

\begin{remark}
\begin{enumerate}
\item Note that when $\phi$ is noncommutative $L$-inducing, then
\[
	P(\phi):=\set{f\in A\colon \duality{\phi}{f}=\norm{f}} 
\]
is precisely the set of positive normal linear functional on the von Neumann algebra $\dual{A}$ associated with $\phi$.
\item In the case of our interest, when $A$ is a Banach algebra and $\phi$ is a character on $A$, if $\phi$ is noncommutative $L$-inducing, then any von Neumann algebraic structure on $\dual{A}$ associated with (this property of) $\phi$ gives $A$ an $F$-algebra structure. In particular, $P(\phi)$ is a subsemigroup of $A$. 
\end{enumerate} 
\end{remark}

We could now modify \cite[Definition 3.5]{Rieffel} as follows.

\begin{definition}\label{noncommutative L'-inducing}
Let $A$ be a Banach algebra, and let $\phi$ be a character on $A$. Then $\phi$ is \emph{noncommutative $L'$-inducing} if it is noncommutative $L$-inducing with an associated von Neumann algebraic structure $\dual{A}_\phi$ on $\dual{A}$ such that
\begin{itemize}
	\item $\abs{fg}_\phi\le_\phi\abs{f}_\phi\abs{g}_\phi$ for every $f,g\in P(\psi)$ where $\psi$ is any noncommutative $L$-inducing character on $A$; 
\end{itemize}
where $\abs{f}_\phi$ is the absolute value of $f$ as a normal linear functional on the von Neumann algebra $\dual{A}_\phi$ and $\le_\phi$ is the order on $A$ as the predual of that von Neumann algebra.
\end{definition}

\begin{remark}\label{remark on noncommutative L'-inducing}
\begin{enumerate}
\item Since $P(\phi)$ is always closed under the multiplication of $A$ when $\phi$ is a noncommutative $L$-inducing character on $A$, in the above definition we are comparing two positive elements in the predual of a von Neumann algebra. 

\item We also remark that in \cite[Definition 3.5]{Rieffel} the inequality $\abs{fg}_\phi\le_\phi\abs{f}_\phi\abs{g}_\phi$ is required to hold true for every $f,g\in A$. But this becomes too restrictive in the noncommutative setting. (However, in \cite{Rieffel}, the inequality is only used for $f,g$ in the same $P(\psi)$ for  some noncommutative $L$-inducing character $\psi$ anyway.)

\item When a character $\phi$ on $A$ is noncommutative $L$-inducing but not actually $L$-inducing (so that an associated von Neumann algebraic structure on $\dual{A}$ is not commutative), there are many different von Neumann algebraic structures on $\dual{A}$ associated to the noncommutative $L$-inducing $\phi$ (but these different von Neumann algebras are all Jordan-isomorphic via the identity map on $\dual{A}$). Now, if $\phi$ is even noncommutative $L'$-inducing, then what we require is that at least one $\dual{A}_\phi$ among all of those von Neumann algebraic structure on $\dual{A}$ associated with $\phi$ will have the property specified in Definition \ref{noncommutative L'-inducing}. 
\end{enumerate}
\end{remark}

The relation between the properties noncommutative $L$-inducing and noncommutative $L'$-inducing for characters may better be clarified by the following.

\begin{lemma}\label{relation of various inducing properties}
Let $A$ be a Banach algebra. Suppose that $A$ admits at least one character that is noncommutative $L'$-inducing. Then every noncommutative $L$-inducing character of $A$ is noncommutative $L'$-inducing. 
\end{lemma}
\begin{proof}
Denote by $D$ the collection of all noncommutative $L$-inducing characters of $A$. For each character $u\in D$, we denote by $\dual{A}_u$ the dual of $A$ with a von Neumann algebra structure associated with the noncommutative $L$-inducing property of $u$. Let $e\in D$ be the specified noncommutative $L'$-inducing character. We shall simply consider $\dual{A}$ as the von Neumann algebra $\dual{A}_e$, where in this particular case, we require further that $\dual{A}_e$ has a von Neumann algebraic structure associated with the noncommutative $L'$-inducing property of $e$. We shall write $e=1$. We also write $\abs{\cdot}$ and $\le$ instead of $\abs{\cdot}_e$ and $\le_e$, respectively.

Take $u\in D$. The identity mapping $\dual{A}_u\to \dual{A}$ is an isometry between two von Neumann algebras, and so it follows from \cite{Kadison51} that $u$, being the identity in $\dual{A}_u$, is a unitary in $\dual{A}$. For each $f\in A$, it is easy to see that
\[
	f\in P(u)\quad\textrm{if and only if}\quad u\cdot f\ge 0\,,
\]
where $\cdot$ is the module action of the von Neumann algebra $\dual{A}$ on $A$. In other words, for each $f,g\ge 0$, then $u^*\cdot f, u^*\cdot g\in P(u)$ with $\abs{u^*\cdot f}=f$ and $\abs{u^*\cdot g}=g$, and so, the additional noncommutative $L'$-inducing condition for $e=1$ then implies that
\[
	u\cdot [(u^*\cdot f)(u^*\cdot g)]=\abs{(u^*\cdot f)(u^*\cdot g)}\le fg;
\]
the equality holds since $(u^*\cdot f)(u^*\cdot g)\in P(u)$ too. However, we have
\begin{align*}
	\duality{1}{u\cdot [(u^*\cdot f)(u^*\cdot g)]}&=\duality{u}{(u^*\cdot f)(u^*\cdot g)}\\
	&=\duality{u}{u^*\cdot f}\duality{u}{u^*\cdot g}\\
	&=\duality{u^*u}{f}\duality{u^*u}{g}=\duality{1}{fg}\,,
\end{align*}
where we use the fact that $u$ and $1$ are characters on $A$. This forces
\[
	u\cdot [(u^*\cdot f)(u^*\cdot g)]= fg,
\]
i.e., $(u^*\cdot f)(u^*\cdot g)=u^*\cdot (fg)$. Thus $T_{u^*}:A\to A$ is a homomorphism. The next claim will then imply that $u^*\in D$.

We \emph{claim} that, for each $u\in \Ucal(\dual{A})$, if the map $T_u: f\mapsto u\cdot f,\ A\to A,$ is a homomorphism, then $u\in D$. Indeed,  $u\in\spectrum(A)$, because $u=\dual{T_u}(1)$ is the transpose of the character $1$ by a homomorphism on $A$. It is also easy to see that $\dual{A}_u$, which is the Banach space $\dual{A}$ with product $x\bullet y:=xu^*y$ and adjoint operation $x^\sharp:=ux^*u$, is a von Neumann algebra, whose unit is $u$, and so $u$ is noncommutative $L$-inducing.

Now, take $u,v\in D$. Applying the foregoing to $u^*, v^*$, which we now know also belong to $D$, we see that the maps $T_u$ and $T_v$ are homomorphisms on $A$. It follows that $T_{uv}=T_uT_v$ is also a homomorphism on $A$. So, by the claim above, we obtain $uv\in D$. Thus $D$ is a subgroup of $\Ucal(\dual{A})$.

Finally, to prove that each $u\in D$ is noncommutative $L'$-inducing, take another $v\in D$ and $f,g\in P(v)$. We give $\dual{A}_u$ the von Neumann algebraic product $\bullet$ and adjoint $\sharp$ as above. We then see that $(u^*v)\cdot (fg), (u^*v)\cdot f, (u^*v)\cdot g\in P(u)$, and so 
\begin{align*}
	\abs{fg}_u=(u^*v)\cdot (fg)=[(u^*v)\cdot f][(u^*v)\cdot g]=\abs{f}_u\abs{g}_u\,,
\end{align*} 
where the second equality is because $u^*v\in D$, and so $T_{u^*v}:A\to A$ is a homomorphism. Thus $u$ is noncommutative $L'$-inducing. 
\end{proof}

For every locally compact group $G$, every character of $\Fa(G)$ is noncommutative $L'$-inducing. Indeed, $\dual{\Fa(G)}=\VN(G)$ is a \wstar-algebra and the spectrum of $\Fa(G)$, considered as a subset of its dual, is $\set{\lambda_G(s)\colon s\in G}$ (see \cite{Eymard}). Moreover, the mapping $f\mapsto \lambda(s)\cdot f, \Fa(G)\to \Fa(G),$ is the right-translation on $\Fa(G)$ by an element $s\in G$, which is an automorphism of $\Fa(G)$. Thus it follows (from the proof of the above proposition) that every character of $\Fa(G)$ is noncommutative $L'$-inducing.

\bigskip

Conversely, we have the following extension of \cite[Theorem A]{Rieffel}; note that, by Lemma \ref{relation of various inducing properties}, we could state the next result with a slightly milder assumption.

\begin{theorem}\label{a characterisation of Fourier algebras 5}
Let $A$ be a commutative semisimple Banach algebra such that
\begin{enumerate}
	\item every character of $A$ is noncommutative $L'$-inducing, and
	\item $A$ is Tauberian and approximately contains a nontrivial real element. 
\end{enumerate}
Then $A\cong \Fa(G)$ for some locally compact group $G$.
\end{theorem}
\begin{proof}
Fix an $e\in\spectrum(A)$, and simply consider $\dual{A}$ as the von Neumann algebra associated with the noncommutative $L'$-inducing property of $e$, and write $e=1$. Then $A$ is an $F$-algebra. Moroever, by the first assumption, as shown in the proof of Lemma \ref{relation of various inducing properties}, the set $\spectrum(A)$ is in fact a subgroup of $\Ucal(\dual{A})$. The result then follows from Corollary \ref{a characterisation of Fourier algebras 3}.
\end{proof}

For compact groups, we have the following. 

\begin{corollary}\label{a characterisation of Fourier algebras 5, compact}
Let $A$ be a unital commutative semisimple Banach algebra such that every character of $A$ is noncommutative $L'$-inducing. Then $A\cong \Fa(G)$ for some compact group $G$. \enproof
\end{corollary}


\section{Characterisations of Fourier--Stieltjes algebras}
\label{Fourier--Stieltjes algebras}

\noindent We now turn to our characterisations of Fourier--Stieltjes algebras. As discussed, unless our locally compact group $G$ turns out to be amenable, we shall only be able to characterise a class of algebras that lie between $\FSa_r(G)$ and $\FSa(G)$. This class is that of the closed translation-invariant Eberlein subalgebras of $\FSa(G)$ that separate points of $G$, whose properties are summarised below.

\begin{theorem}\label{a class of Fourier--Stieltjes algebras}
Let $G$ be a locally compact group, and let $\Scal$ be a \emph{thick} sub-semidual of $G_d$. Then:
\begin{enumerate}
	\item $\FSa_\Scal^c(G)$ is a {\rm(}weak$^*${\rm)} closed translation-invariant Eberlein subalgebra of $\FSa(G)$, which is also conjugation-closed if $\Scal$ is actually a subdual of $G_d$. 
	\item The dual of $\FSa_\Scal^c(G)$ is naturally a von Neumann algebra.
\end{enumerate}
If, moreover, either one of the following additional conditions 
\begin{itemize}
	\item[(a)] $\Scal$ is actually a subdual of $G_d$, or
	\item[(b)] $G$ contains an open subgroup $H$ with $H_d$ amenable,
\end{itemize}
holds, then:
\begin{enumerate}
	\addtocounter{enumi}{2}
	\item $\FSa_r(G)\subseteq \FSa_\Scal^c(G)$.
	\item $G$ is naturally mapped, homomorphically and homeomorphically, onto a unitary subgroup of $\dual{\FSa_\Scal^c(G)}$ endowed with weak$^*$-topology.
	\item The image of $G$ via this map is also the set of noncommutative $L'$-inducing characters  on $\FSa_\Scal^c(G)$.
\end{enumerate}
\end{theorem}
\begin{remark}
\begin{enumerate}
\item We remark that when in fact $G_d$ is amenable, then both assumptions (a) and (b) hold. 

\item We do not know whether (iii) always holds for every thick sub-semidual $\Scal$ of $G_d$. If that turns out to be the case, then (iv) and (v) also always hold for every thick sub-semidual $\Scal$ of $G_d$.
\end{enumerate} 
\end{remark}

\begin{proof}
Most has been proved: (i) is from Theorem \ref{when translation-invariant subalgebra is Eberlein}, and (ii) and most of (iv) could be seen from the proof of Lemma \ref{Eberlein cover}.

The next part of the proof requires either (a) or (b) holds. In either of these cases, (iii) follows from either Lemma \ref{when a subdual is thick} or Lemma \ref{when a sub-semidual is thick}.

To complete the proof of (iv), it is sufficient to show that $\FSa_\Scal^c(G)$ determines the topology of $G$. But this is obvious since $\Fa(G)\subseteq \FSa_\Scal^c(G)$ by (iii).

To prove (v), notice that the image of $G$ is easily seen to be contained in  the set of noncommutative $L$-inducing characters of $B:=\FSa_\Scal^c(G)$. On the other hand, the set of noncommutative $L$-inducing characters of $B$ is contained in $\spectrum(B)\cap \Ucal(\dual{B})$ (cf. the proof of Lemma \ref{relation of various inducing properties}). But $\spectrum(B)\cap \Ucal(\dual{B})$ is precisely the image of $G$ via the natural map: since $B\supseteq \Fa(G)$ by (iii) again, this follows by the same argument used in the proof of \cite[Theorem 1.(i)]{Walter72} (see also \cite[Lemma 3.3]{LL93}). Thus $G$ is naturally identified with the set of noncommutative $L$-inducing characters of $B$. But then the $\dual{B}$-module actions by those characters on $B$ are automorphism of $B$, and so they must be indeed noncommutative $L'$-inducing (cf. again the proof of Lemma \ref{relation of various inducing properties}). 
\end{proof}

This theorem also provides the converse of our characterisation results below. 

\begin{theorem}\label{a characterisation of Fourier--Stieltjes algebras}
Let $B$ be an $F$-algebra that is also a conjugation-closed Eberlein algebra on a topological space $\Omega$. Denote by $\eta:\Omega\to\spectrum(B)$ the natural mapping. Suppose also that:
\begin{enumerate}
	\item $\eta(\Omega)$ is locally compact in $\spectrum(B)$, self-adjoint, and contains $1$;
	\item $\Omega$ is a group and $B$ is left translation-invariant; 
	\item $\norm{\sum_{i=1}^m \alpha_i L_{s_i}:B\to B}\le 1$ whenever $\alpha_i\in\complexs$ and $s_i\in \Omega$ with
		\[	
			\abs{\sum_{i=1}^m\alpha_i f(s_i)}\le \norm{f}\quad(f\in B)\,.
		\]
\end{enumerate}
Then $B\cong \FSa_\Scal^c(G)$ for  some thick subdual $\Scal$ of $\widehat{G_d}$ of a locally compact group $G$.
\end{theorem}
\begin{proof}
Set $G:=\eta(\Omega)$. Note that condition (i) implies that $B$ is an $F$-algebra. As in the proof of Theorem \ref{a characterisation of Fourier algebras}, we see that $G$ is a (locally compact) topological subgroup of the unitary group of $\dual{B}$ with the weak$^*$-topology, and $\eta:\Omega\to G$ is a continuous group homomorphism, where the multiplications on $\Omega$ and $\dual{B}$ may need to be modified if necessary. Also, identifying $B$ with $\widehat{B}|_G$, $B\subseteq \FSa(G)$ and the inclusion $\Phi:B\to \FSa(G)$ is bounded by \v{S}hilov's  theorem. Set $W^*(G):=\bidual{\C^*(G)}$, the universal von Neumann algebra of $G$ (see \cite{Walter72}). Consider $\dual{\Phi}:W^*(G)\to \dual{B}$. Then $\dual{\Phi}$ is a (normal) $*$-homomorphism with a weak$^*$-dense range in $\dual{B}$. It follows that $\dual{\Phi}(W^*(G))=\dual{B}$, and so $\Phi$ is an isometric isomorphism from $B$ onto $\Phi(B)$, i.e. $B$ is closed in $\FSa(G)$. The results in \S\ref{Eberlein subalgebras of Fourier--Stieltjes algebras} then conclude the proof.
\end{proof}

\begin{remark}\label{a characterisation of Fourier--Stieltjes algebras: remark}
If we omit the assumption on the conjugation-closedness from the hypothesis of the previous theorem, we could still conclude that $B\cong \FSa_\Scal^c(G)$ for  some thick sub-semidual $\Scal$ of $\widehat{G_d}$ of a locally compact group $G$. However, in this case, we no longer know whether the converse holds true and whether or not $\FSa_r(G)\subseteq B$. 
\end{remark}

To capture precisely $\FSa(G)$, we need (to make sure that) $G$ to be \emph{amenable}. 

\begin{definition}
Let $\Omega$ be a \emph{left-topological group}, i.e. $\Omega$ is a group with a topology such that for each $s\in \Omega$, the map $t\mapsto st, \Omega\to \Omega,$ is continuous. Then $\Omega$ is \emph{left-amenable} if there is a \emph{left-invariant mean} on $\C^b(\Omega)$, i.e. a mean $m$ on $\C^b(\Omega)$ such that $m=m\circ L_s$ for every $s\in \Omega$.
\end{definition}

\begin{corollary}\label{a characterisation of Fourier--Stieltjes algebras 2}
Assuming as in the theorem. If in addition $\Omega$ is a left-amenable left-topological group. Then $B\cong \FSa(G)$ for an amenable locally compact group $G$.
\end{corollary}
\begin{proof}
We have seen that $\FSa_r(G)\subseteq B\subseteq \FSa(G)$ for some locally compact group $G$. Moreover, it is easy to see from the above proof that a left-invariant mean on $\C^b(\Omega)$ provided by the left-amenability of $\Omega$ induces a left-invariant mean on $\C^b(G)$, and so $G$ is amenable. Hence $B= \FSa_r(G)=\FSa(G)$.
\end{proof}

Counter-parts of Theorems \ref{a characterisation of Fourier algebras 4} and \ref{a characterisation of Fourier algebras 4b} for Fourier--Stieltjes algebras are below (cf. \cite[Theorem 7]{Walter74}).

\begin{theorem}\label{a characterisation of Fourier--Stieltjes algebras 4}
Let $B$ be an $F$-algebra that is also a conjugation-closed Eberlein algebra on a topological space $\Omega$. Denote by $\eta:\Omega\to\spectrum(B)$ the natural mapping. Suppose also that:
\begin{enumerate}
	\item $\eta(\Omega)$ is locally compact in $\spectrum(B)$ and contains $1$;
	\item $\Omega$ is a group, $B$ is left translation-invariant, and, for each $s\in\Omega$, 
		the automorphism $L_s$ is dual for $B$.
\end{enumerate}
Then $B\cong \FSa_\Scal^c(G)$ for  some thick subdual $\Scal$ of $\widehat{G_d}$ of a locally compact group $G$.
\end{theorem}
\begin{proof}
This is proved in a similar way as for Theorem \ref{a characterisation of Fourier algebras 4}.
\end{proof}

An extra condition as in Corollary \ref{a characterisation of Fourier--Stieltjes algebras 2} could be added to the hypothesis of the above to give a characterisation of $\FSa(G)$ for amenable locally compact groups $G$.

\begin{theorem}\label{a characterisation of Fourier--Stieltjes algebras 4b}
Let $B$ be a  commutative $F$-algebra. Suppose that  the orbit of $1$ by some dual group of $B$ is a locally compact subspace $\Omega$ of $\spectrum(B)$ that is separating for $B$ and such that $\widehat{B}|_{\Omega}$ is a conjugation-closed Eberlein algebra on $\Omega$. Then $B\cong \FSa_\Scal^c(G)$ for  some thick subdual $\Scal$ of $\widehat{G_d}$ of a locally compact group $G$. \enproof
\end{theorem}

Below is an extension of the characterisation in \cite[Theorem 7.2]{Rieffel} of measure algebras on locally compact abelian groups to an almost a characterisation of Fourier--Stieltjes algebras on locally compact groups:

\begin{theorem}\label{a characterisation of Fourier--Stieltjes algebras 5}
Let $B$ be a commutative Banach algebra, and let $D$ be the collection of all noncommutative $L'$-inducing characters on $B$. Suppose that
\begin{enumerate}
	\item $D$ is a separating family of linear functionals of $B$;
	\item $D$ is locally compact in the weak$^*$-topology;
	\item $\widehat{B}|_D$ is a conjugation-closed Eberlein algebra on $D$. 
\end{enumerate}
Then $B\cong \FSa_\Scal^c(G)$ for  some thick subdual $\Scal$ of $\widehat{G_d}$ of a locally compact group $G$. 
\end{theorem}
\begin{proof}
As in the proof of Theorem \ref{a characterisation of Fourier algebras 5}, by fixing an element in $D$, which will be denoted by $1$, we give $\dual{B}$ a von Neumann algebraic structure whose unit is this element $1$. By Lemma \ref{relation of various inducing properties} and its proof, we see that  $D$ is a subgroup of $\Ucal(\dual{B})$ that consists of all noncommutative $L$-inducing characters on $B$. The result will then follows from  Theorem \ref{a characterisation of Fourier--Stieltjes algebras}. 
\end{proof}

For the class of locally compact groups $G$ with $G_d$ amenable, we could remove the conjugation-closed assumptions from some of the above results  and, at the same time, obtain characterisations of exactly $\FSa(G)$.

\begin{theorem}\label{a characterisation of Fourier--Stieltjes algebras, discretely amenable}
Let $B$ be an $F$-algebra that is also an Eberlein algebra on a topological space $\Omega$. Denote by $\eta:\Omega\to\spectrum(B)$ the natural mapping. Suppose also that:
\begin{enumerate}
	\item $\eta(\Omega)$ is locally compact in $\spectrum(B)$, self-adjoint, and contains $1$;
	\item $\Omega$, without topology, is an amenable group and $B$ is left translation-invariant; 
	\item $\norm{\sum_{i=1}^m \alpha_i L_{s_i}:B\to B}\le 1$ whenever $\alpha_i\in\complexs$ and $s_i\in \Omega$ with
		\[	
			\abs{\sum_{i=1}^m\alpha_i f(s_i)}\le \norm{f}\quad(f\in B)\,.
		\]
\end{enumerate}
Then $B\cong \FSa(G)$ for  some locally compact group $G$ with $G_d$ amenable. \enproof
\end{theorem}

\begin{theorem}
Let $B$ be an $F$-algebra that is also an Eberlein algebra on a topological space $\Omega$. Denote by $\eta:\Omega\to\spectrum(B)$ the natural mapping. Suppose also that:
\begin{enumerate}
	\item $\eta(\Omega)$ is locally compact in $\spectrum(B)$ and contains $1$;
	\item $\Omega$, without topology, is an amenable group, $B$ is left translation-invariant, and, for each $s\in\Omega$, 	the automorphism $L_s$ is dual for $B$.
\end{enumerate}
Then $B\cong \FSa(G)$ for  some locally compact group $G$ with $G_d$ amenable. \enproof
\end{theorem}

\begin{theorem}\label{a characterisation of Fourier--Stieltjes algebras 4d}
Let $B$ be a  commutative $F$-algebra. Suppose that the orbit of $1$ by some amenable dual group of $B$ is a locally compact subspace $\Omega$ of $\spectrum(B)$ that is separating for $B$. Suppose also that $\widehat{B}|_{\Omega}$ is a conjugation-closed  Eberlein algebra on $\Omega$. Then $B\cong \FSa(G)$ for  some locally compact group $G$ with $G_d$ amenable. \enproof
\end{theorem}

\section{Further remarks and open problems}

\label{Further remarks and open problems}

\begin{problem}\label{when A=A(G) problem}
Let $G$ be a locally compact group, and let $A$ be a closed translation-invariant Tauberian subalgebra of $\Fa(G)$ with $\spectrum(A)=G$. Must $A=\Fa(G)$?
\end{problem}

It has been shown that the answer to this is `Yes' in  \cite{Rieffel}  if $G$ is abelian or in \S\ref{Invariant subalgebras of Fourier algebras} if the connected component of $G$ is compact or  if $A$ satisfies an additional condition that it approximately contains a nontrivial real function. But we do not know the answer in general or even when $G$ is amenable.

It follows from Lemma \ref{when A=A(G), compact-convergence dense case} that to prove that such algebra $A=\Fa(G)$ it is sufficient to show that $A$ is dense in $\C_0(G)$. One condition would imply this is that $A$ is \emph{operator amenable}, a concept introduced by Ruan in \cite{Ruan} where he also shows that $\Fa(G)$ is operator amenable if and only if $G$ is amenable. Our assertion follows from Runde's  extension \cite{Runde} of \v{S}e\u{\i}nberg's theorem \cite{Seinberg}. Thus this could be a possible line of attack of Problem \ref{when A=A(G) problem} in the case when $G$ is amenable. 

A related result is a recent theorem of Crann and Neufang \cite{CN} which shows that $G$ is amenable if and only if $\VN(G)$, which is the dual of $\Fa(G)$, is injective in a certain natural module category. This may helps in resolving our problem in the amenable case along the line above, since, after all, the dual $\dual{A}$ of our algebra $A$ always ends up being a quotient of $\VN(G)$ by a normal ideal. Moreover, this type of condition on $\dual{A}$ could be a neat addition to the hypotheses of our results that would give us characterisations of the Fourier algebras of \emph{amenable} locally compact groups. In our opinion, this type of condition on $\dual{A}$ is more desirable than, says, operator amenability of $A$, since for the more special case of \emph{abelian} locally compact groups we could simply require that $\dual{A}$ is commutative.

\begin{problem}\label{when Eberlien contains A(G)}
Let $G$ be a locally compact group, and let $B$ be a closed translation-invariant Eberlein subalgebra of $\FSa(G)$ that separates points of $G$. Must $B$ contain $\Fa(G)$?
\end{problem}

Equivalently, this could be stated in the following way:

\smallskip

\noindent \textbf{Open problem \ref{when Eberlien contains A(G)}'.} Let $G$ be a locally compact group, and let $B$ be a translation-invariant subalgebra of $\FSa(G)$ that separates points of $G$. Take $\phi\in\Fa(G)$. Must $\phi$ be a pointwise limit of functions in $B\cap P(G)$?

\smallskip

The answer to this is `Yes' when $G$ contains an open subgroup $H$ with $H_d$ amenable (Theorem \ref{a class of Fourier--Stieltjes algebras}). But we do not know the answer for other groups.

\begin{problem}\label{when Eberlein is conjugation closed}
Let $G$ be a discrete group, and let $B$ be a weak$^*$-closed translation-invariant subalgebra of $\FSa(G)$. Must $B$ be  conjugation-closed? 

Equivalently, must a sub-semidual of a discrete group be already a subdual?
\end{problem}

Corollary \ref{weak*-closed invariant subalgebra of B(G) when G is discrete 2} shows that if $G$ is (discrete and) amenable, then the answer to this question is affirmative. We suspect  that this is not the case in general, but we do not have a counter-example. This problem could also be stated in the following form for locally compact groups.

\smallskip

\noindent \textbf{Open problem \ref{when Eberlein is conjugation closed}'.} Let $G$ be a locally compact group, and let $B$ be a closed translation-invariant Eberlein subalgebra of $\FSa(G)$. Must $B$ be conjugation-closed?

\smallskip

If this is true, then combining with \cite[Theorem 1.3]{BLSchl} it would resolve Problem \ref{when Eberlien contains A(G)} in the affirmative. 

\begin{problem}
What extra assumption could be added in Theorem \ref{a characterisation of Fourier--Stieltjes algebras 5} to guarantee that $G$ is amenable, so that we have a characterisation of precisely $\FSa(G)$ in that case? 
\end{problem}

Preferably, the additional assumption should be a condition on $\dual{A}$ that would be automatically satisfied when $\dual{A}$ is commutative. Some form of module injectivity as considered in \cite{CN} could be a possibility.

\begin{problem}
Let $G$ be a locally compact group. Is $\FSa_r(G)$ always an Eberlein algebra on $G$? If not, then for which $G$, is $\FSa_r(G)$ an Eberlein algebra on $G$? This is true if $G$ is amenable or discrete. Could it be true when $G$ has an amenable open subgroup?
\end{problem}

This is equivalent to asking whether a continuous function  $\phi$ that is a pointwise limit of functions in $\Fa(G)\cap P(G)$ must belong to $\FSa_r(G)$. 

\begin{problem}
Let $S$ be a topological semigroup with involution, and let $\FSa(S)$ be the Fourier--Stieltjes algebra of $S$ as defined in \cite{Lau78}. For which $S$ does $\FSa(S)\cong\FSa(G)$ for some locally compact group $G$? The same problem when $S$ is a topological group.
\end{problem}

\end{document}